\documentclass{article}%
\usepackage{amsmath}%
\usepackage{amsfonts}%
\usepackage{amssymb}%
\usepackage{graphicx}
\usepackage[latin1]{inputenc}

\newtheorem{theorem}{Theorem}[section]

\newtheorem{definition}{Definition}[section]

\newtheorem{lemma}[theorem]{Lemma}

\newtheorem{proposition}[theorem]{Proposition}

\newenvironment{proof}[1][Proof]{\textbf{#1.} }{\ \rule{0.5em}{0.5em}}

\begin{document}

\title{Commutator length of annulus diffeomorphisms}
\author{E. Militon}
\date{\today}
\maketitle

\begin{abstract}
We study the group $\mathrm{Diff}_{0}^{r}(\mathbb{A})$ of $C^{r}$-diffeomorphisms of the closed annulus that are isotopic to the identity. We show that, for $r \neq 3$, the linear space of homogeneous quasi-morphisms on the group $\mathrm{Diff}_{0}^{r}(\mathbb{A})$ is one dimensional. Therefore, the commutator length on this group is (stably) unbounded. In particular, this provides an example of a manifold whose diffeomorphisms group is unbounded in the sense of Burago, Ivanov and Polterovich. 
\end{abstract}

\section{Introduction}

Let $M$ be a manifold. For $r$ in $\mathbb{N} \cup \left\{ \infty \right\}$, denote by $\mathrm{Diff}_{0}^{r}(M)$ the identity component of the group of $C^{r}$-diffeomorphism of $M$. We also write $\mathrm{Homeo}_{0}(M)=\mathrm{Diff}_{0}^{0}(M)$. We study algebraic properties of these groups.

A commutator in $\mathrm{Diff}_{0}^{r}(M)$ is an element of the form $[f,g]=f\circ g \circ f^{-1} \circ g^{-1}$. It is known that, for $r \neq dim(M)$ and $r \neq dim(M)+1$  every element in $\mathrm{Diff}_{0}^{r}(M)$ can be written as a product of commutators. A question naturally arises: how many commutators shall we need to write a given diffeomorphism as a product of commutators? For an element $f$ in $\mathrm{Diff}_{0}^{r}(M)$, the \textit{commutator length} $cl_{r}(f)$ is the minimal number of commutators needed to write $f$ as a product of commutators. Burago, Ivanov and Polterovich showed in \cite{BIP} that, when the manifold $M$ is a sphere or compact and three-dimensional without boundary, the commutator length is bounded (by $4$ in the case of the sphere and $10$ in the case of three-dimensional manifolds). They also exhibited a wide class of open manifolds (portable manifolds), including $\mathbb{R}^{n}$, for which the commutator length is bounded by $2$. Tsuboi generalized these results for odd-dimensional compact manifolds and gave a better bound: the commutator length is bounded by $6$ in those cases.

The \textit{support} $supp(f)$ of a homeomorphism $f$ of $M$ is defined to be the closure of the set:
$$\left\{ x \in M,f(x)\neq x \right\}.$$
A homeomorphism $f$ is said to be \textit{supported in a ball} if there exists a topological embedding $i:\mathbb{B}\rightarrow M$ of the closed unit ball $\mathbb{B}$ of the same dimension as $M$ such that $supp(f) \subset i(\mathbb{B})$ and $i(\mathbb{B}) \cap \partial M$ is homeomorphic to an open cell. Like for commutators, every element $f$ in $\mathrm{Diff}_{0}^{r}(M)$ can be written as a product of diffeomorphisms in $\mathrm{Diff}_{0}^{r}(M)$ supported in balls and we define $\mathrm{Frag}_{r}(f)$ to be the minimal number of diffeomorphisms in such a decomposition. Burago, Ivanov, Polterovich and Tsuboi showed also that the fragmentation norm is bounded in the case of odd-dimensional manifolds and of open portable manifolds.

The commutator length and the fragmentation norm are two examples of the more general notion of conjugation-invariant norm on a group introduced by Burago, Ivanov and Polterovich in \cite{BIP}. They proved that the fragmentation norm plays a crucial role: every conjugation-invariant norm on $\mathrm{Diff}_{0}^{r}(M)$ is bounded if and only if the fragmentation norm is bounded.

In this paper, we consider the case of the closed annulus $\mathbb{A}$. We prove some estimates on the commutator length and the fragmentation norm, which implie in particular that they are unbounded, contrarily to the above examples \footnote{Observe in particular that, for a manifold $M$ with interior $\mathrm{int}(M)$, the boundedness of $\mathrm{Diff}_{0}^{r}(\partial M)$ and of $\mathrm{Diff}_{0}^{r}(\mathrm{int}(M))$ does not necessarily implie that $\mathrm{Diff}_{0}^{r}(M)$ is bounded (compare to \cite{Afu} ).}. For a diffeomorphism $f$ in $\mathrm{Diff}_{0}^{r}(\mathbb{A})$, let $\rho(f)$ be the difference of the translation numbers of $f$ on the two boundary components (for a precise definition of $\rho(f)$, see the next section). The following theorem shows that $\rho$ is a quasi-isometry from $\mathrm{Diff}_{0}^{r}(M)$, endowed with the fragmentation norm or the commutator length, to $\mathbb{R}$:

\begin{theorem} \label{qi} Let $r$ be an integer different from $2$ and $3$. For any diffeomorphism $f$ in $\mathrm{Diff}_{0}^{r}(\mathbb{A})$, 
$$ \left| \frac{| \rho(f) |}{4}- cl_{r}(f) \right| \leq 12$$
and
$$| \rho(f)-\mathrm{Frag}_{r}(f)| \leq 40.$$
\end{theorem}

In fact, the map $\rho$ is a quasi-morphism (this notion will be defined in the next section) on the group $\mathrm{Diff}_{0}^{r}(\mathbb{A})$ and is essentially the only quasi-morphism on this group:

\begin{theorem} \label{qm} If $r \neq 2, 3$, every homogeneous quasi-morphism on $\mathrm{Diff}_{0}^{r}(\mathbb{A})$ is colinear to $\rho$.
\end{theorem}

\underline{Remark}: for $r=2$ or $r=3$, the result holds if we replace $\mathrm{Diff}_{0}^{r}(\mathbb{A})$ by $[\mathrm{Diff}_{0}^{r}(\mathbb{A}),\mathrm{Diff}_{0}^{r}(\mathbb{A})]$, the subgroup generated by commutators. Whether the equality $[\mathrm{Diff}_{0}^{3}(\mathbb{A}),\mathrm{Diff}_{0}^{3}(\mathbb{A})]=\mathrm{Diff}_{0}^{3}(\mathbb{A})$ holds or not is still an open question.

In the next section, we state a proposition which give fine estimates on the commutator length. We also show that this proposition implie the first part of Theorem \ref{qi} and Theorem \ref{qm}. The third section is devoted to the proof of this proposition. In the fourth section, we prove an analogous proposition on the fragmentation norm which implies the second part of Theorem \ref{qi}. Finally, in the last section, we discuss generalizations of these results to other surfaces. We will prove the following result:

\begin{theorem}
Let $M$ be a surface with non-empty boundary which is different from the disk and $r$ be a number in $\mathbb{N} \cup \left\{\infty\right\}$. Then the group $\mathrm{Diff}_{0}^{r}(M)$ admits non-trivial homogeneous quasi-morphisms. In particular, this group is stably unbounded.
\end{theorem}

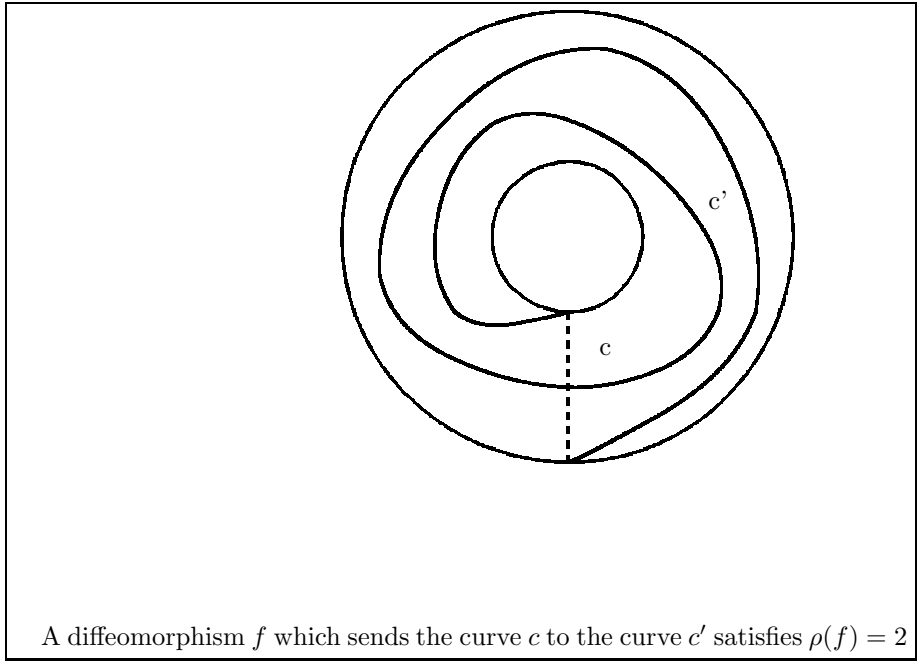
\framebox{~~~\parbox{0.95\textwidth}{
\ifx\JPicScale\undefined\def\JPicScale{1}\fi
\unitlength \JPicScale mm
\begin{picture}(100,80)(0,0)
\linethickness{0.3mm}
\put(80,49.75){\line(0,1){0.51}}
\multiput(79.97,50.76)(0.03,-0.51){1}{\line(0,-1){0.51}}
\multiput(79.92,51.26)(0.05,-0.5){1}{\line(0,-1){0.5}}
\multiput(79.84,51.76)(0.08,-0.5){1}{\line(0,-1){0.5}}
\multiput(79.74,52.26)(0.1,-0.5){1}{\line(0,-1){0.5}}
\multiput(79.61,52.75)(0.13,-0.49){1}{\line(0,-1){0.49}}
\multiput(79.46,53.23)(0.15,-0.48){1}{\line(0,-1){0.48}}
\multiput(79.29,53.71)(0.18,-0.48){1}{\line(0,-1){0.48}}
\multiput(79.09,54.18)(0.1,-0.23){2}{\line(0,-1){0.23}}
\multiput(78.86,54.63)(0.11,-0.23){2}{\line(0,-1){0.23}}
\multiput(78.62,55.07)(0.12,-0.22){2}{\line(0,-1){0.22}}
\multiput(78.35,55.5)(0.13,-0.21){2}{\line(0,-1){0.21}}
\multiput(78.06,55.92)(0.14,-0.21){2}{\line(0,-1){0.21}}
\multiput(77.75,56.32)(0.1,-0.13){3}{\line(0,-1){0.13}}
\multiput(77.42,56.7)(0.11,-0.13){3}{\line(0,-1){0.13}}
\multiput(77.07,57.07)(0.12,-0.12){3}{\line(0,-1){0.12}}
\multiput(76.7,57.42)(0.12,-0.12){3}{\line(1,0){0.12}}
\multiput(76.32,57.75)(0.13,-0.11){3}{\line(1,0){0.13}}
\multiput(75.92,58.06)(0.13,-0.1){3}{\line(1,0){0.13}}
\multiput(75.5,58.35)(0.21,-0.14){2}{\line(1,0){0.21}}
\multiput(75.07,58.62)(0.21,-0.13){2}{\line(1,0){0.21}}
\multiput(74.63,58.86)(0.22,-0.12){2}{\line(1,0){0.22}}
\multiput(74.18,59.09)(0.23,-0.11){2}{\line(1,0){0.23}}
\multiput(73.71,59.29)(0.23,-0.1){2}{\line(1,0){0.23}}
\multiput(73.23,59.46)(0.48,-0.18){1}{\line(1,0){0.48}}
\multiput(72.75,59.61)(0.48,-0.15){1}{\line(1,0){0.48}}
\multiput(72.26,59.74)(0.49,-0.13){1}{\line(1,0){0.49}}
\multiput(71.76,59.84)(0.5,-0.1){1}{\line(1,0){0.5}}
\multiput(71.26,59.92)(0.5,-0.08){1}{\line(1,0){0.5}}
\multiput(70.76,59.97)(0.5,-0.05){1}{\line(1,0){0.5}}
\multiput(70.25,60)(0.51,-0.03){1}{\line(1,0){0.51}}
\put(69.75,60){\line(1,0){0.51}}
\multiput(69.24,59.97)(0.51,0.03){1}{\line(1,0){0.51}}
\multiput(68.74,59.92)(0.5,0.05){1}{\line(1,0){0.5}}
\multiput(68.24,59.84)(0.5,0.08){1}{\line(1,0){0.5}}
\multiput(67.74,59.74)(0.5,0.1){1}{\line(1,0){0.5}}
\multiput(67.25,59.61)(0.49,0.13){1}{\line(1,0){0.49}}
\multiput(66.77,59.46)(0.48,0.15){1}{\line(1,0){0.48}}
\multiput(66.29,59.29)(0.48,0.18){1}{\line(1,0){0.48}}
\multiput(65.82,59.09)(0.23,0.1){2}{\line(1,0){0.23}}
\multiput(65.37,58.86)(0.23,0.11){2}{\line(1,0){0.23}}
\multiput(64.93,58.62)(0.22,0.12){2}{\line(1,0){0.22}}
\multiput(64.5,58.35)(0.21,0.13){2}{\line(1,0){0.21}}
\multiput(64.08,58.06)(0.21,0.14){2}{\line(1,0){0.21}}
\multiput(63.68,57.75)(0.13,0.1){3}{\line(1,0){0.13}}
\multiput(63.3,57.42)(0.13,0.11){3}{\line(1,0){0.13}}
\multiput(62.93,57.07)(0.12,0.12){3}{\line(1,0){0.12}}
\multiput(62.58,56.7)(0.12,0.12){3}{\line(0,1){0.12}}
\multiput(62.25,56.32)(0.11,0.13){3}{\line(0,1){0.13}}
\multiput(61.94,55.92)(0.1,0.13){3}{\line(0,1){0.13}}
\multiput(61.65,55.5)(0.14,0.21){2}{\line(0,1){0.21}}
\multiput(61.38,55.07)(0.13,0.21){2}{\line(0,1){0.21}}
\multiput(61.14,54.63)(0.12,0.22){2}{\line(0,1){0.22}}
\multiput(60.91,54.18)(0.11,0.23){2}{\line(0,1){0.23}}
\multiput(60.71,53.71)(0.1,0.23){2}{\line(0,1){0.23}}
\multiput(60.54,53.23)(0.18,0.48){1}{\line(0,1){0.48}}
\multiput(60.39,52.75)(0.15,0.48){1}{\line(0,1){0.48}}
\multiput(60.26,52.26)(0.13,0.49){1}{\line(0,1){0.49}}
\multiput(60.16,51.76)(0.1,0.5){1}{\line(0,1){0.5}}
\multiput(60.08,51.26)(0.08,0.5){1}{\line(0,1){0.5}}
\multiput(60.03,50.76)(0.05,0.5){1}{\line(0,1){0.5}}
\multiput(60,50.25)(0.03,0.51){1}{\line(0,1){0.51}}
\put(60,49.75){\line(0,1){0.51}}
\multiput(60,49.75)(0.03,-0.51){1}{\line(0,-1){0.51}}
\multiput(60.03,49.24)(0.05,-0.5){1}{\line(0,-1){0.5}}
\multiput(60.08,48.74)(0.08,-0.5){1}{\line(0,-1){0.5}}
\multiput(60.16,48.24)(0.1,-0.5){1}{\line(0,-1){0.5}}
\multiput(60.26,47.74)(0.13,-0.49){1}{\line(0,-1){0.49}}
\multiput(60.39,47.25)(0.15,-0.48){1}{\line(0,-1){0.48}}
\multiput(60.54,46.77)(0.18,-0.48){1}{\line(0,-1){0.48}}
\multiput(60.71,46.29)(0.1,-0.23){2}{\line(0,-1){0.23}}
\multiput(60.91,45.82)(0.11,-0.23){2}{\line(0,-1){0.23}}
\multiput(61.14,45.37)(0.12,-0.22){2}{\line(0,-1){0.22}}
\multiput(61.38,44.93)(0.13,-0.21){2}{\line(0,-1){0.21}}
\multiput(61.65,44.5)(0.14,-0.21){2}{\line(0,-1){0.21}}
\multiput(61.94,44.08)(0.1,-0.13){3}{\line(0,-1){0.13}}
\multiput(62.25,43.68)(0.11,-0.13){3}{\line(0,-1){0.13}}
\multiput(62.58,43.3)(0.12,-0.12){3}{\line(0,-1){0.12}}
\multiput(62.93,42.93)(0.12,-0.12){3}{\line(1,0){0.12}}
\multiput(63.3,42.58)(0.13,-0.11){3}{\line(1,0){0.13}}
\multiput(63.68,42.25)(0.13,-0.1){3}{\line(1,0){0.13}}
\multiput(64.08,41.94)(0.21,-0.14){2}{\line(1,0){0.21}}
\multiput(64.5,41.65)(0.21,-0.13){2}{\line(1,0){0.21}}
\multiput(64.93,41.38)(0.22,-0.12){2}{\line(1,0){0.22}}
\multiput(65.37,41.14)(0.23,-0.11){2}{\line(1,0){0.23}}
\multiput(65.82,40.91)(0.23,-0.1){2}{\line(1,0){0.23}}
\multiput(66.29,40.71)(0.48,-0.18){1}{\line(1,0){0.48}}
\multiput(66.77,40.54)(0.48,-0.15){1}{\line(1,0){0.48}}
\multiput(67.25,40.39)(0.49,-0.13){1}{\line(1,0){0.49}}
\multiput(67.74,40.26)(0.5,-0.1){1}{\line(1,0){0.5}}
\multiput(68.24,40.16)(0.5,-0.08){1}{\line(1,0){0.5}}
\multiput(68.74,40.08)(0.5,-0.05){1}{\line(1,0){0.5}}
\multiput(69.24,40.03)(0.51,-0.03){1}{\line(1,0){0.51}}
\put(69.75,40){\line(1,0){0.51}}
\multiput(70.25,40)(0.51,0.03){1}{\line(1,0){0.51}}
\multiput(70.76,40.03)(0.5,0.05){1}{\line(1,0){0.5}}
\multiput(71.26,40.08)(0.5,0.08){1}{\line(1,0){0.5}}
\multiput(71.76,40.16)(0.5,0.1){1}{\line(1,0){0.5}}
\multiput(72.26,40.26)(0.49,0.13){1}{\line(1,0){0.49}}
\multiput(72.75,40.39)(0.48,0.15){1}{\line(1,0){0.48}}
\multiput(73.23,40.54)(0.48,0.18){1}{\line(1,0){0.48}}
\multiput(73.71,40.71)(0.23,0.1){2}{\line(1,0){0.23}}
\multiput(74.18,40.91)(0.23,0.11){2}{\line(1,0){0.23}}
\multiput(74.63,41.14)(0.22,0.12){2}{\line(1,0){0.22}}
\multiput(75.07,41.38)(0.21,0.13){2}{\line(1,0){0.21}}
\multiput(75.5,41.65)(0.21,0.14){2}{\line(1,0){0.21}}
\multiput(75.92,41.94)(0.13,0.1){3}{\line(1,0){0.13}}
\multiput(76.32,42.25)(0.13,0.11){3}{\line(1,0){0.13}}
\multiput(76.7,42.58)(0.12,0.12){3}{\line(1,0){0.12}}
\multiput(77.07,42.93)(0.12,0.12){3}{\line(0,1){0.12}}
\multiput(77.42,43.3)(0.11,0.13){3}{\line(0,1){0.13}}
\multiput(77.75,43.68)(0.1,0.13){3}{\line(0,1){0.13}}
\multiput(78.06,44.08)(0.14,0.21){2}{\line(0,1){0.21}}
\multiput(78.35,44.5)(0.13,0.21){2}{\line(0,1){0.21}}
\multiput(78.62,44.93)(0.12,0.22){2}{\line(0,1){0.22}}
\multiput(78.86,45.37)(0.11,0.23){2}{\line(0,1){0.23}}
\multiput(79.09,45.82)(0.1,0.23){2}{\line(0,1){0.23}}
\multiput(79.29,46.29)(0.18,0.48){1}{\line(0,1){0.48}}
\multiput(79.46,46.77)(0.15,0.48){1}{\line(0,1){0.48}}
\multiput(79.61,47.25)(0.13,0.49){1}{\line(0,1){0.49}}
\multiput(79.74,47.74)(0.1,0.5){1}{\line(0,1){0.5}}
\multiput(79.84,48.24)(0.08,0.5){1}{\line(0,1){0.5}}
\multiput(79.92,48.74)(0.05,0.5){1}{\line(0,1){0.5}}
\multiput(79.97,49.24)(0.03,0.51){1}{\line(0,1){0.51}}

\linethickness{0.3mm}
\put(100,49.75){\line(0,1){0.5}}
\multiput(99.99,50.75)(0.01,-0.5){1}{\line(0,-1){0.5}}
\multiput(99.97,51.25)(0.02,-0.5){1}{\line(0,-1){0.5}}
\multiput(99.95,51.75)(0.03,-0.5){1}{\line(0,-1){0.5}}
\multiput(99.92,52.25)(0.03,-0.5){1}{\line(0,-1){0.5}}
\multiput(99.87,52.75)(0.04,-0.5){1}{\line(0,-1){0.5}}
\multiput(99.82,53.25)(0.05,-0.5){1}{\line(0,-1){0.5}}
\multiput(99.76,53.75)(0.06,-0.5){1}{\line(0,-1){0.5}}
\multiput(99.7,54.25)(0.07,-0.5){1}{\line(0,-1){0.5}}
\multiput(99.62,54.74)(0.08,-0.5){1}{\line(0,-1){0.5}}
\multiput(99.54,55.24)(0.08,-0.49){1}{\line(0,-1){0.49}}
\multiput(99.45,55.73)(0.09,-0.49){1}{\line(0,-1){0.49}}
\multiput(99.35,56.22)(0.1,-0.49){1}{\line(0,-1){0.49}}
\multiput(99.24,56.71)(0.11,-0.49){1}{\line(0,-1){0.49}}
\multiput(99.12,57.2)(0.12,-0.49){1}{\line(0,-1){0.49}}
\multiput(99,57.68)(0.12,-0.49){1}{\line(0,-1){0.49}}
\multiput(98.87,58.17)(0.13,-0.48){1}{\line(0,-1){0.48}}
\multiput(98.73,58.65)(0.14,-0.48){1}{\line(0,-1){0.48}}
\multiput(98.58,59.13)(0.15,-0.48){1}{\line(0,-1){0.48}}
\multiput(98.42,59.6)(0.16,-0.48){1}{\line(0,-1){0.48}}
\multiput(98.26,60.08)(0.16,-0.47){1}{\line(0,-1){0.47}}
\multiput(98.08,60.55)(0.17,-0.47){1}{\line(0,-1){0.47}}
\multiput(97.9,61.02)(0.09,-0.23){2}{\line(0,-1){0.23}}
\multiput(97.72,61.48)(0.09,-0.23){2}{\line(0,-1){0.23}}
\multiput(97.52,61.94)(0.1,-0.23){2}{\line(0,-1){0.23}}
\multiput(97.32,62.4)(0.1,-0.23){2}{\line(0,-1){0.23}}
\multiput(97.11,62.85)(0.11,-0.23){2}{\line(0,-1){0.23}}
\multiput(96.89,63.31)(0.11,-0.23){2}{\line(0,-1){0.23}}
\multiput(96.66,63.75)(0.11,-0.22){2}{\line(0,-1){0.22}}
\multiput(96.43,64.2)(0.12,-0.22){2}{\line(0,-1){0.22}}
\multiput(96.19,64.64)(0.12,-0.22){2}{\line(0,-1){0.22}}
\multiput(95.94,65.07)(0.12,-0.22){2}{\line(0,-1){0.22}}
\multiput(95.68,65.5)(0.13,-0.22){2}{\line(0,-1){0.22}}
\multiput(95.42,65.93)(0.13,-0.21){2}{\line(0,-1){0.21}}
\multiput(95.15,66.35)(0.13,-0.21){2}{\line(0,-1){0.21}}
\multiput(94.87,66.77)(0.14,-0.21){2}{\line(0,-1){0.21}}
\multiput(94.59,67.18)(0.14,-0.21){2}{\line(0,-1){0.21}}
\multiput(94.3,67.59)(0.15,-0.2){2}{\line(0,-1){0.2}}
\multiput(94,68)(0.15,-0.2){2}{\line(0,-1){0.2}}
\multiput(93.7,68.4)(0.1,-0.13){3}{\line(0,-1){0.13}}
\multiput(93.39,68.79)(0.1,-0.13){3}{\line(0,-1){0.13}}
\multiput(93.07,69.18)(0.11,-0.13){3}{\line(0,-1){0.13}}
\multiput(92.75,69.56)(0.11,-0.13){3}{\line(0,-1){0.13}}
\multiput(92.42,69.94)(0.11,-0.13){3}{\line(0,-1){0.13}}
\multiput(92.08,70.31)(0.11,-0.12){3}{\line(0,-1){0.12}}
\multiput(91.74,70.67)(0.11,-0.12){3}{\line(0,-1){0.12}}
\multiput(91.39,71.04)(0.12,-0.12){3}{\line(0,-1){0.12}}
\multiput(91.04,71.39)(0.12,-0.12){3}{\line(1,0){0.12}}
\multiput(90.67,71.74)(0.12,-0.12){3}{\line(1,0){0.12}}
\multiput(90.31,72.08)(0.12,-0.11){3}{\line(1,0){0.12}}
\multiput(89.94,72.42)(0.12,-0.11){3}{\line(1,0){0.12}}
\multiput(89.56,72.75)(0.13,-0.11){3}{\line(1,0){0.13}}
\multiput(89.18,73.07)(0.13,-0.11){3}{\line(1,0){0.13}}
\multiput(88.79,73.39)(0.13,-0.11){3}{\line(1,0){0.13}}
\multiput(88.4,73.7)(0.13,-0.1){3}{\line(1,0){0.13}}
\multiput(88,74)(0.13,-0.1){3}{\line(1,0){0.13}}
\multiput(87.59,74.3)(0.2,-0.15){2}{\line(1,0){0.2}}
\multiput(87.18,74.59)(0.2,-0.15){2}{\line(1,0){0.2}}
\multiput(86.77,74.87)(0.21,-0.14){2}{\line(1,0){0.21}}
\multiput(86.35,75.15)(0.21,-0.14){2}{\line(1,0){0.21}}
\multiput(85.93,75.42)(0.21,-0.13){2}{\line(1,0){0.21}}
\multiput(85.5,75.68)(0.21,-0.13){2}{\line(1,0){0.21}}
\multiput(85.07,75.94)(0.22,-0.13){2}{\line(1,0){0.22}}
\multiput(84.64,76.19)(0.22,-0.12){2}{\line(1,0){0.22}}
\multiput(84.2,76.43)(0.22,-0.12){2}{\line(1,0){0.22}}
\multiput(83.75,76.66)(0.22,-0.12){2}{\line(1,0){0.22}}
\multiput(83.31,76.89)(0.22,-0.11){2}{\line(1,0){0.22}}
\multiput(82.85,77.11)(0.23,-0.11){2}{\line(1,0){0.23}}
\multiput(82.4,77.32)(0.23,-0.11){2}{\line(1,0){0.23}}
\multiput(81.94,77.52)(0.23,-0.1){2}{\line(1,0){0.23}}
\multiput(81.48,77.72)(0.23,-0.1){2}{\line(1,0){0.23}}
\multiput(81.02,77.9)(0.23,-0.09){2}{\line(1,0){0.23}}
\multiput(80.55,78.08)(0.23,-0.09){2}{\line(1,0){0.23}}
\multiput(80.08,78.26)(0.47,-0.17){1}{\line(1,0){0.47}}
\multiput(79.6,78.42)(0.47,-0.16){1}{\line(1,0){0.47}}
\multiput(79.13,78.58)(0.48,-0.16){1}{\line(1,0){0.48}}
\multiput(78.65,78.73)(0.48,-0.15){1}{\line(1,0){0.48}}
\multiput(78.17,78.87)(0.48,-0.14){1}{\line(1,0){0.48}}
\multiput(77.68,79)(0.48,-0.13){1}{\line(1,0){0.48}}
\multiput(77.2,79.12)(0.49,-0.12){1}{\line(1,0){0.49}}
\multiput(76.71,79.24)(0.49,-0.12){1}{\line(1,0){0.49}}
\multiput(76.22,79.35)(0.49,-0.11){1}{\line(1,0){0.49}}
\multiput(75.73,79.45)(0.49,-0.1){1}{\line(1,0){0.49}}
\multiput(75.24,79.54)(0.49,-0.09){1}{\line(1,0){0.49}}
\multiput(74.74,79.62)(0.49,-0.08){1}{\line(1,0){0.49}}
\multiput(74.25,79.7)(0.5,-0.08){1}{\line(1,0){0.5}}
\multiput(73.75,79.76)(0.5,-0.07){1}{\line(1,0){0.5}}
\multiput(73.25,79.82)(0.5,-0.06){1}{\line(1,0){0.5}}
\multiput(72.75,79.87)(0.5,-0.05){1}{\line(1,0){0.5}}
\multiput(72.25,79.92)(0.5,-0.04){1}{\line(1,0){0.5}}
\multiput(71.75,79.95)(0.5,-0.03){1}{\line(1,0){0.5}}
\multiput(71.25,79.97)(0.5,-0.03){1}{\line(1,0){0.5}}
\multiput(70.75,79.99)(0.5,-0.02){1}{\line(1,0){0.5}}
\multiput(70.25,80)(0.5,-0.01){1}{\line(1,0){0.5}}
\put(69.75,80){\line(1,0){0.5}}
\multiput(69.25,79.99)(0.5,0.01){1}{\line(1,0){0.5}}
\multiput(68.75,79.97)(0.5,0.02){1}{\line(1,0){0.5}}
\multiput(68.25,79.95)(0.5,0.03){1}{\line(1,0){0.5}}
\multiput(67.75,79.92)(0.5,0.03){1}{\line(1,0){0.5}}
\multiput(67.25,79.87)(0.5,0.04){1}{\line(1,0){0.5}}
\multiput(66.75,79.82)(0.5,0.05){1}{\line(1,0){0.5}}
\multiput(66.25,79.76)(0.5,0.06){1}{\line(1,0){0.5}}
\multiput(65.75,79.7)(0.5,0.07){1}{\line(1,0){0.5}}
\multiput(65.26,79.62)(0.5,0.08){1}{\line(1,0){0.5}}
\multiput(64.76,79.54)(0.49,0.08){1}{\line(1,0){0.49}}
\multiput(64.27,79.45)(0.49,0.09){1}{\line(1,0){0.49}}
\multiput(63.78,79.35)(0.49,0.1){1}{\line(1,0){0.49}}
\multiput(63.29,79.24)(0.49,0.11){1}{\line(1,0){0.49}}
\multiput(62.8,79.12)(0.49,0.12){1}{\line(1,0){0.49}}
\multiput(62.32,79)(0.49,0.12){1}{\line(1,0){0.49}}
\multiput(61.83,78.87)(0.48,0.13){1}{\line(1,0){0.48}}
\multiput(61.35,78.73)(0.48,0.14){1}{\line(1,0){0.48}}
\multiput(60.87,78.58)(0.48,0.15){1}{\line(1,0){0.48}}
\multiput(60.4,78.42)(0.48,0.16){1}{\line(1,0){0.48}}
\multiput(59.92,78.26)(0.47,0.16){1}{\line(1,0){0.47}}
\multiput(59.45,78.08)(0.47,0.17){1}{\line(1,0){0.47}}
\multiput(58.98,77.9)(0.23,0.09){2}{\line(1,0){0.23}}
\multiput(58.52,77.72)(0.23,0.09){2}{\line(1,0){0.23}}
\multiput(58.06,77.52)(0.23,0.1){2}{\line(1,0){0.23}}
\multiput(57.6,77.32)(0.23,0.1){2}{\line(1,0){0.23}}
\multiput(57.15,77.11)(0.23,0.11){2}{\line(1,0){0.23}}
\multiput(56.69,76.89)(0.23,0.11){2}{\line(1,0){0.23}}
\multiput(56.25,76.66)(0.22,0.11){2}{\line(1,0){0.22}}
\multiput(55.8,76.43)(0.22,0.12){2}{\line(1,0){0.22}}
\multiput(55.36,76.19)(0.22,0.12){2}{\line(1,0){0.22}}
\multiput(54.93,75.94)(0.22,0.12){2}{\line(1,0){0.22}}
\multiput(54.5,75.68)(0.22,0.13){2}{\line(1,0){0.22}}
\multiput(54.07,75.42)(0.21,0.13){2}{\line(1,0){0.21}}
\multiput(53.65,75.15)(0.21,0.13){2}{\line(1,0){0.21}}
\multiput(53.23,74.87)(0.21,0.14){2}{\line(1,0){0.21}}
\multiput(52.82,74.59)(0.21,0.14){2}{\line(1,0){0.21}}
\multiput(52.41,74.3)(0.2,0.15){2}{\line(1,0){0.2}}
\multiput(52,74)(0.2,0.15){2}{\line(1,0){0.2}}
\multiput(51.6,73.7)(0.13,0.1){3}{\line(1,0){0.13}}
\multiput(51.21,73.39)(0.13,0.1){3}{\line(1,0){0.13}}
\multiput(50.82,73.07)(0.13,0.11){3}{\line(1,0){0.13}}
\multiput(50.44,72.75)(0.13,0.11){3}{\line(1,0){0.13}}
\multiput(50.06,72.42)(0.13,0.11){3}{\line(1,0){0.13}}
\multiput(49.69,72.08)(0.12,0.11){3}{\line(1,0){0.12}}
\multiput(49.33,71.74)(0.12,0.11){3}{\line(1,0){0.12}}
\multiput(48.96,71.39)(0.12,0.12){3}{\line(1,0){0.12}}
\multiput(48.61,71.04)(0.12,0.12){3}{\line(1,0){0.12}}
\multiput(48.26,70.67)(0.12,0.12){3}{\line(0,1){0.12}}
\multiput(47.92,70.31)(0.11,0.12){3}{\line(0,1){0.12}}
\multiput(47.58,69.94)(0.11,0.12){3}{\line(0,1){0.12}}
\multiput(47.25,69.56)(0.11,0.13){3}{\line(0,1){0.13}}
\multiput(46.93,69.18)(0.11,0.13){3}{\line(0,1){0.13}}
\multiput(46.61,68.79)(0.11,0.13){3}{\line(0,1){0.13}}
\multiput(46.3,68.4)(0.1,0.13){3}{\line(0,1){0.13}}
\multiput(46,68)(0.1,0.13){3}{\line(0,1){0.13}}
\multiput(45.7,67.59)(0.15,0.2){2}{\line(0,1){0.2}}
\multiput(45.41,67.18)(0.15,0.2){2}{\line(0,1){0.2}}
\multiput(45.13,66.77)(0.14,0.21){2}{\line(0,1){0.21}}
\multiput(44.85,66.35)(0.14,0.21){2}{\line(0,1){0.21}}
\multiput(44.58,65.93)(0.13,0.21){2}{\line(0,1){0.21}}
\multiput(44.32,65.5)(0.13,0.21){2}{\line(0,1){0.21}}
\multiput(44.06,65.07)(0.13,0.22){2}{\line(0,1){0.22}}
\multiput(43.81,64.64)(0.12,0.22){2}{\line(0,1){0.22}}
\multiput(43.57,64.2)(0.12,0.22){2}{\line(0,1){0.22}}
\multiput(43.34,63.75)(0.12,0.22){2}{\line(0,1){0.22}}
\multiput(43.11,63.31)(0.11,0.22){2}{\line(0,1){0.22}}
\multiput(42.89,62.85)(0.11,0.23){2}{\line(0,1){0.23}}
\multiput(42.68,62.4)(0.11,0.23){2}{\line(0,1){0.23}}
\multiput(42.48,61.94)(0.1,0.23){2}{\line(0,1){0.23}}
\multiput(42.28,61.48)(0.1,0.23){2}{\line(0,1){0.23}}
\multiput(42.1,61.02)(0.09,0.23){2}{\line(0,1){0.23}}
\multiput(41.92,60.55)(0.09,0.23){2}{\line(0,1){0.23}}
\multiput(41.74,60.08)(0.17,0.47){1}{\line(0,1){0.47}}
\multiput(41.58,59.6)(0.16,0.47){1}{\line(0,1){0.47}}
\multiput(41.42,59.13)(0.16,0.48){1}{\line(0,1){0.48}}
\multiput(41.27,58.65)(0.15,0.48){1}{\line(0,1){0.48}}
\multiput(41.13,58.17)(0.14,0.48){1}{\line(0,1){0.48}}
\multiput(41,57.68)(0.13,0.48){1}{\line(0,1){0.48}}
\multiput(40.88,57.2)(0.12,0.49){1}{\line(0,1){0.49}}
\multiput(40.76,56.71)(0.12,0.49){1}{\line(0,1){0.49}}
\multiput(40.65,56.22)(0.11,0.49){1}{\line(0,1){0.49}}
\multiput(40.55,55.73)(0.1,0.49){1}{\line(0,1){0.49}}
\multiput(40.46,55.24)(0.09,0.49){1}{\line(0,1){0.49}}
\multiput(40.38,54.74)(0.08,0.49){1}{\line(0,1){0.49}}
\multiput(40.3,54.25)(0.08,0.5){1}{\line(0,1){0.5}}
\multiput(40.24,53.75)(0.07,0.5){1}{\line(0,1){0.5}}
\multiput(40.18,53.25)(0.06,0.5){1}{\line(0,1){0.5}}
\multiput(40.13,52.75)(0.05,0.5){1}{\line(0,1){0.5}}
\multiput(40.08,52.25)(0.04,0.5){1}{\line(0,1){0.5}}
\multiput(40.05,51.75)(0.03,0.5){1}{\line(0,1){0.5}}
\multiput(40.03,51.25)(0.03,0.5){1}{\line(0,1){0.5}}
\multiput(40.01,50.75)(0.02,0.5){1}{\line(0,1){0.5}}
\multiput(40,50.25)(0.01,0.5){1}{\line(0,1){0.5}}
\put(40,49.75){\line(0,1){0.5}}
\multiput(40,49.75)(0.01,-0.5){1}{\line(0,-1){0.5}}
\multiput(40.01,49.25)(0.02,-0.5){1}{\line(0,-1){0.5}}
\multiput(40.03,48.75)(0.03,-0.5){1}{\line(0,-1){0.5}}
\multiput(40.05,48.25)(0.03,-0.5){1}{\line(0,-1){0.5}}
\multiput(40.08,47.75)(0.04,-0.5){1}{\line(0,-1){0.5}}
\multiput(40.13,47.25)(0.05,-0.5){1}{\line(0,-1){0.5}}
\multiput(40.18,46.75)(0.06,-0.5){1}{\line(0,-1){0.5}}
\multiput(40.24,46.25)(0.07,-0.5){1}{\line(0,-1){0.5}}
\multiput(40.3,45.75)(0.08,-0.5){1}{\line(0,-1){0.5}}
\multiput(40.38,45.26)(0.08,-0.49){1}{\line(0,-1){0.49}}
\multiput(40.46,44.76)(0.09,-0.49){1}{\line(0,-1){0.49}}
\multiput(40.55,44.27)(0.1,-0.49){1}{\line(0,-1){0.49}}
\multiput(40.65,43.78)(0.11,-0.49){1}{\line(0,-1){0.49}}
\multiput(40.76,43.29)(0.12,-0.49){1}{\line(0,-1){0.49}}
\multiput(40.88,42.8)(0.12,-0.49){1}{\line(0,-1){0.49}}
\multiput(41,42.32)(0.13,-0.48){1}{\line(0,-1){0.48}}
\multiput(41.13,41.83)(0.14,-0.48){1}{\line(0,-1){0.48}}
\multiput(41.27,41.35)(0.15,-0.48){1}{\line(0,-1){0.48}}
\multiput(41.42,40.87)(0.16,-0.48){1}{\line(0,-1){0.48}}
\multiput(41.58,40.4)(0.16,-0.47){1}{\line(0,-1){0.47}}
\multiput(41.74,39.92)(0.17,-0.47){1}{\line(0,-1){0.47}}
\multiput(41.92,39.45)(0.09,-0.23){2}{\line(0,-1){0.23}}
\multiput(42.1,38.98)(0.09,-0.23){2}{\line(0,-1){0.23}}
\multiput(42.28,38.52)(0.1,-0.23){2}{\line(0,-1){0.23}}
\multiput(42.48,38.06)(0.1,-0.23){2}{\line(0,-1){0.23}}
\multiput(42.68,37.6)(0.11,-0.23){2}{\line(0,-1){0.23}}
\multiput(42.89,37.15)(0.11,-0.23){2}{\line(0,-1){0.23}}
\multiput(43.11,36.69)(0.11,-0.22){2}{\line(0,-1){0.22}}
\multiput(43.34,36.25)(0.12,-0.22){2}{\line(0,-1){0.22}}
\multiput(43.57,35.8)(0.12,-0.22){2}{\line(0,-1){0.22}}
\multiput(43.81,35.36)(0.12,-0.22){2}{\line(0,-1){0.22}}
\multiput(44.06,34.93)(0.13,-0.22){2}{\line(0,-1){0.22}}
\multiput(44.32,34.5)(0.13,-0.21){2}{\line(0,-1){0.21}}
\multiput(44.58,34.07)(0.13,-0.21){2}{\line(0,-1){0.21}}
\multiput(44.85,33.65)(0.14,-0.21){2}{\line(0,-1){0.21}}
\multiput(45.13,33.23)(0.14,-0.21){2}{\line(0,-1){0.21}}
\multiput(45.41,32.82)(0.15,-0.2){2}{\line(0,-1){0.2}}
\multiput(45.7,32.41)(0.15,-0.2){2}{\line(0,-1){0.2}}
\multiput(46,32)(0.1,-0.13){3}{\line(0,-1){0.13}}
\multiput(46.3,31.6)(0.1,-0.13){3}{\line(0,-1){0.13}}
\multiput(46.61,31.21)(0.11,-0.13){3}{\line(0,-1){0.13}}
\multiput(46.93,30.82)(0.11,-0.13){3}{\line(0,-1){0.13}}
\multiput(47.25,30.44)(0.11,-0.13){3}{\line(0,-1){0.13}}
\multiput(47.58,30.06)(0.11,-0.12){3}{\line(0,-1){0.12}}
\multiput(47.92,29.69)(0.11,-0.12){3}{\line(0,-1){0.12}}
\multiput(48.26,29.33)(0.12,-0.12){3}{\line(0,-1){0.12}}
\multiput(48.61,28.96)(0.12,-0.12){3}{\line(1,0){0.12}}
\multiput(48.96,28.61)(0.12,-0.12){3}{\line(1,0){0.12}}
\multiput(49.33,28.26)(0.12,-0.11){3}{\line(1,0){0.12}}
\multiput(49.69,27.92)(0.12,-0.11){3}{\line(1,0){0.12}}
\multiput(50.06,27.58)(0.13,-0.11){3}{\line(1,0){0.13}}
\multiput(50.44,27.25)(0.13,-0.11){3}{\line(1,0){0.13}}
\multiput(50.82,26.93)(0.13,-0.11){3}{\line(1,0){0.13}}
\multiput(51.21,26.61)(0.13,-0.1){3}{\line(1,0){0.13}}
\multiput(51.6,26.3)(0.13,-0.1){3}{\line(1,0){0.13}}
\multiput(52,26)(0.2,-0.15){2}{\line(1,0){0.2}}
\multiput(52.41,25.7)(0.2,-0.15){2}{\line(1,0){0.2}}
\multiput(52.82,25.41)(0.21,-0.14){2}{\line(1,0){0.21}}
\multiput(53.23,25.13)(0.21,-0.14){2}{\line(1,0){0.21}}
\multiput(53.65,24.85)(0.21,-0.13){2}{\line(1,0){0.21}}
\multiput(54.07,24.58)(0.21,-0.13){2}{\line(1,0){0.21}}
\multiput(54.5,24.32)(0.22,-0.13){2}{\line(1,0){0.22}}
\multiput(54.93,24.06)(0.22,-0.12){2}{\line(1,0){0.22}}
\multiput(55.36,23.81)(0.22,-0.12){2}{\line(1,0){0.22}}
\multiput(55.8,23.57)(0.22,-0.12){2}{\line(1,0){0.22}}
\multiput(56.25,23.34)(0.22,-0.11){2}{\line(1,0){0.22}}
\multiput(56.69,23.11)(0.23,-0.11){2}{\line(1,0){0.23}}
\multiput(57.15,22.89)(0.23,-0.11){2}{\line(1,0){0.23}}
\multiput(57.6,22.68)(0.23,-0.1){2}{\line(1,0){0.23}}
\multiput(58.06,22.48)(0.23,-0.1){2}{\line(1,0){0.23}}
\multiput(58.52,22.28)(0.23,-0.09){2}{\line(1,0){0.23}}
\multiput(58.98,22.1)(0.23,-0.09){2}{\line(1,0){0.23}}
\multiput(59.45,21.92)(0.47,-0.17){1}{\line(1,0){0.47}}
\multiput(59.92,21.74)(0.47,-0.16){1}{\line(1,0){0.47}}
\multiput(60.4,21.58)(0.48,-0.16){1}{\line(1,0){0.48}}
\multiput(60.87,21.42)(0.48,-0.15){1}{\line(1,0){0.48}}
\multiput(61.35,21.27)(0.48,-0.14){1}{\line(1,0){0.48}}
\multiput(61.83,21.13)(0.48,-0.13){1}{\line(1,0){0.48}}
\multiput(62.32,21)(0.49,-0.12){1}{\line(1,0){0.49}}
\multiput(62.8,20.88)(0.49,-0.12){1}{\line(1,0){0.49}}
\multiput(63.29,20.76)(0.49,-0.11){1}{\line(1,0){0.49}}
\multiput(63.78,20.65)(0.49,-0.1){1}{\line(1,0){0.49}}
\multiput(64.27,20.55)(0.49,-0.09){1}{\line(1,0){0.49}}
\multiput(64.76,20.46)(0.49,-0.08){1}{\line(1,0){0.49}}
\multiput(65.26,20.38)(0.5,-0.08){1}{\line(1,0){0.5}}
\multiput(65.75,20.3)(0.5,-0.07){1}{\line(1,0){0.5}}
\multiput(66.25,20.24)(0.5,-0.06){1}{\line(1,0){0.5}}
\multiput(66.75,20.18)(0.5,-0.05){1}{\line(1,0){0.5}}
\multiput(67.25,20.13)(0.5,-0.04){1}{\line(1,0){0.5}}
\multiput(67.75,20.08)(0.5,-0.03){1}{\line(1,0){0.5}}
\multiput(68.25,20.05)(0.5,-0.03){1}{\line(1,0){0.5}}
\multiput(68.75,20.03)(0.5,-0.02){1}{\line(1,0){0.5}}
\multiput(69.25,20.01)(0.5,-0.01){1}{\line(1,0){0.5}}
\put(69.75,20){\line(1,0){0.5}}
\multiput(70.25,20)(0.5,0.01){1}{\line(1,0){0.5}}
\multiput(70.75,20.01)(0.5,0.02){1}{\line(1,0){0.5}}
\multiput(71.25,20.03)(0.5,0.03){1}{\line(1,0){0.5}}
\multiput(71.75,20.05)(0.5,0.03){1}{\line(1,0){0.5}}
\multiput(72.25,20.08)(0.5,0.04){1}{\line(1,0){0.5}}
\multiput(72.75,20.13)(0.5,0.05){1}{\line(1,0){0.5}}
\multiput(73.25,20.18)(0.5,0.06){1}{\line(1,0){0.5}}
\multiput(73.75,20.24)(0.5,0.07){1}{\line(1,0){0.5}}
\multiput(74.25,20.3)(0.5,0.08){1}{\line(1,0){0.5}}
\multiput(74.74,20.38)(0.49,0.08){1}{\line(1,0){0.49}}
\multiput(75.24,20.46)(0.49,0.09){1}{\line(1,0){0.49}}
\multiput(75.73,20.55)(0.49,0.1){1}{\line(1,0){0.49}}
\multiput(76.22,20.65)(0.49,0.11){1}{\line(1,0){0.49}}
\multiput(76.71,20.76)(0.49,0.12){1}{\line(1,0){0.49}}
\multiput(77.2,20.88)(0.49,0.12){1}{\line(1,0){0.49}}
\multiput(77.68,21)(0.48,0.13){1}{\line(1,0){0.48}}
\multiput(78.17,21.13)(0.48,0.14){1}{\line(1,0){0.48}}
\multiput(78.65,21.27)(0.48,0.15){1}{\line(1,0){0.48}}
\multiput(79.13,21.42)(0.48,0.16){1}{\line(1,0){0.48}}
\multiput(79.6,21.58)(0.47,0.16){1}{\line(1,0){0.47}}
\multiput(80.08,21.74)(0.47,0.17){1}{\line(1,0){0.47}}
\multiput(80.55,21.92)(0.23,0.09){2}{\line(1,0){0.23}}
\multiput(81.02,22.1)(0.23,0.09){2}{\line(1,0){0.23}}
\multiput(81.48,22.28)(0.23,0.1){2}{\line(1,0){0.23}}
\multiput(81.94,22.48)(0.23,0.1){2}{\line(1,0){0.23}}
\multiput(82.4,22.68)(0.23,0.11){2}{\line(1,0){0.23}}
\multiput(82.85,22.89)(0.23,0.11){2}{\line(1,0){0.23}}
\multiput(83.31,23.11)(0.22,0.11){2}{\line(1,0){0.22}}
\multiput(83.75,23.34)(0.22,0.12){2}{\line(1,0){0.22}}
\multiput(84.2,23.57)(0.22,0.12){2}{\line(1,0){0.22}}
\multiput(84.64,23.81)(0.22,0.12){2}{\line(1,0){0.22}}
\multiput(85.07,24.06)(0.22,0.13){2}{\line(1,0){0.22}}
\multiput(85.5,24.32)(0.21,0.13){2}{\line(1,0){0.21}}
\multiput(85.93,24.58)(0.21,0.13){2}{\line(1,0){0.21}}
\multiput(86.35,24.85)(0.21,0.14){2}{\line(1,0){0.21}}
\multiput(86.77,25.13)(0.21,0.14){2}{\line(1,0){0.21}}
\multiput(87.18,25.41)(0.2,0.15){2}{\line(1,0){0.2}}
\multiput(87.59,25.7)(0.2,0.15){2}{\line(1,0){0.2}}
\multiput(88,26)(0.13,0.1){3}{\line(1,0){0.13}}
\multiput(88.4,26.3)(0.13,0.1){3}{\line(1,0){0.13}}
\multiput(88.79,26.61)(0.13,0.11){3}{\line(1,0){0.13}}
\multiput(89.18,26.93)(0.13,0.11){3}{\line(1,0){0.13}}
\multiput(89.56,27.25)(0.13,0.11){3}{\line(1,0){0.13}}
\multiput(89.94,27.58)(0.12,0.11){3}{\line(1,0){0.12}}
\multiput(90.31,27.92)(0.12,0.11){3}{\line(1,0){0.12}}
\multiput(90.67,28.26)(0.12,0.12){3}{\line(1,0){0.12}}
\multiput(91.04,28.61)(0.12,0.12){3}{\line(1,0){0.12}}
\multiput(91.39,28.96)(0.12,0.12){3}{\line(0,1){0.12}}
\multiput(91.74,29.33)(0.11,0.12){3}{\line(0,1){0.12}}
\multiput(92.08,29.69)(0.11,0.12){3}{\line(0,1){0.12}}
\multiput(92.42,30.06)(0.11,0.13){3}{\line(0,1){0.13}}
\multiput(92.75,30.44)(0.11,0.13){3}{\line(0,1){0.13}}
\multiput(93.07,30.82)(0.11,0.13){3}{\line(0,1){0.13}}
\multiput(93.39,31.21)(0.1,0.13){3}{\line(0,1){0.13}}
\multiput(93.7,31.6)(0.1,0.13){3}{\line(0,1){0.13}}
\multiput(94,32)(0.15,0.2){2}{\line(0,1){0.2}}
\multiput(94.3,32.41)(0.15,0.2){2}{\line(0,1){0.2}}
\multiput(94.59,32.82)(0.14,0.21){2}{\line(0,1){0.21}}
\multiput(94.87,33.23)(0.14,0.21){2}{\line(0,1){0.21}}
\multiput(95.15,33.65)(0.13,0.21){2}{\line(0,1){0.21}}
\multiput(95.42,34.07)(0.13,0.21){2}{\line(0,1){0.21}}
\multiput(95.68,34.5)(0.13,0.22){2}{\line(0,1){0.22}}
\multiput(95.94,34.93)(0.12,0.22){2}{\line(0,1){0.22}}
\multiput(96.19,35.36)(0.12,0.22){2}{\line(0,1){0.22}}
\multiput(96.43,35.8)(0.12,0.22){2}{\line(0,1){0.22}}
\multiput(96.66,36.25)(0.11,0.22){2}{\line(0,1){0.22}}
\multiput(96.89,36.69)(0.11,0.23){2}{\line(0,1){0.23}}
\multiput(97.11,37.15)(0.11,0.23){2}{\line(0,1){0.23}}
\multiput(97.32,37.6)(0.1,0.23){2}{\line(0,1){0.23}}
\multiput(97.52,38.06)(0.1,0.23){2}{\line(0,1){0.23}}
\multiput(97.72,38.52)(0.09,0.23){2}{\line(0,1){0.23}}
\multiput(97.9,38.98)(0.09,0.23){2}{\line(0,1){0.23}}
\multiput(98.08,39.45)(0.17,0.47){1}{\line(0,1){0.47}}
\multiput(98.26,39.92)(0.16,0.47){1}{\line(0,1){0.47}}
\multiput(98.42,40.4)(0.16,0.48){1}{\line(0,1){0.48}}
\multiput(98.58,40.87)(0.15,0.48){1}{\line(0,1){0.48}}
\multiput(98.73,41.35)(0.14,0.48){1}{\line(0,1){0.48}}
\multiput(98.87,41.83)(0.13,0.48){1}{\line(0,1){0.48}}
\multiput(99,42.32)(0.12,0.49){1}{\line(0,1){0.49}}
\multiput(99.12,42.8)(0.12,0.49){1}{\line(0,1){0.49}}
\multiput(99.24,43.29)(0.11,0.49){1}{\line(0,1){0.49}}
\multiput(99.35,43.78)(0.1,0.49){1}{\line(0,1){0.49}}
\multiput(99.45,44.27)(0.09,0.49){1}{\line(0,1){0.49}}
\multiput(99.54,44.76)(0.08,0.49){1}{\line(0,1){0.49}}
\multiput(99.62,45.26)(0.08,0.5){1}{\line(0,1){0.5}}
\multiput(99.7,45.75)(0.07,0.5){1}{\line(0,1){0.5}}
\multiput(99.76,46.25)(0.06,0.5){1}{\line(0,1){0.5}}
\multiput(99.82,46.75)(0.05,0.5){1}{\line(0,1){0.5}}
\multiput(99.87,47.25)(0.04,0.5){1}{\line(0,1){0.5}}
\multiput(99.92,47.75)(0.03,0.5){1}{\line(0,1){0.5}}
\multiput(99.95,48.25)(0.03,0.5){1}{\line(0,1){0.5}}
\multiput(99.97,48.75)(0.02,0.5){1}{\line(0,1){0.5}}
\multiput(99.99,49.25)(0.01,0.5){1}{\line(0,1){0.5}}

\linethickness{0.3mm}
\qbezier(70,20)(71.37,20.26)(81.74,26.04)
\qbezier(81.74,26.04)(92.12,31.82)(95,40)
\qbezier(95,40)(96.64,50.75)(90.87,61.83)
\qbezier(90.87,61.83)(85.1,72.91)(75,75)
\qbezier(75,75)(63.91,75.64)(54.18,65.88)
\qbezier(54.18,65.88)(44.46,56.13)(45,45)
\qbezier(45,45)(46.4,37.86)(54.22,34.08)
\qbezier(54.22,34.08)(62.04,30.3)(70,30)
\qbezier(70,30)(76.02,29.77)(81.97,32.24)
\qbezier(81.97,32.24)(87.92,34.71)(90,40)
\qbezier(90,40)(91.48,45.44)(88.12,50.86)
\qbezier(88.12,50.86)(84.76,56.27)(80,60)
\qbezier(80,60)(75.69,63.55)(70.11,65.53)
\qbezier(70.11,65.53)(64.54,67.51)(60,65)
\qbezier(60,65)(54.49,61.07)(52.82,53.24)
\qbezier(52.82,53.24)(51.14,45.4)(55,40)
\qbezier(55,40)(58,37.49)(63.68,38.57)
\qbezier(63.68,38.57)(69.36,39.64)(70,40)
\linethickness{0.3mm}
\multiput(70,20)(0,1.9){11}{\line(0,1){0.95}}
\put(75,35){\makebox(0,0)[cc]{c}}

\put(90,55){\makebox(0,0)[cc]{c'}}

\end{picture}

A diffeomorphism $f$ which sends the curve $c$ to the curve $c'$ satisfies $\rho (f)=2$}}

\section{Quasimorphisms on $\mathrm{Diff}_{0}^{r}(\mathbb{A})$}

\begin{definition}
Given a group $G$, a homogeneous quasi-morphism on $G$ is a map $q:G \rightarrow \mathbb{R}$ which satisfies:
\begin{enumerate}
\item $\exists D>0, \ \forall a,b \in G, \ \left|q(ab)-q(a)-q(b)\right|\leq D$.
\item $\forall a \in G, \ \forall n \in \mathbb{Z}, q(a^{n})=nq(a)$.
\end{enumerate}
The least constant $D(q)$ which satisfies 1) is called the defect of the quasimorphism $q$.
The trivial quasimorphism is the quasimorphism which maps every element of $G$ to $0$.
\end{definition}

A first connection between commutator length and quasi-morphisms is given by the following classical lemma:

\begin{lemma} \label{qmcl}
Suppose $G$ is perfect. If $G$ admits a non-trivial homogeneous quasi-morphism $q$, then $cl_{G}$ is unbounded. Moreover,
$$ \forall g \in G, |q(g)| \leq 2 cl_{G}(g)D(q).$$
\end{lemma}

\begin{proof}
First, we show that $q$ is conjugation invariant. Indeed, for elements $g$ and $h$ in $G$:
$$|q(hgh^{-1})-q(g)| =\frac{1}{n}|q(hg^{n}h^{-1})-q(g^{n})| \leq \frac{1}{n}D(q)$$
and $q(hgh^{-1})=q(g)$.
Let $g$ be an element of $G$ with commutator length $cl_{G}(g)=n$. Then $g$ can be written the following way:
$$g=[g_{1},h_{1}][g_{2},h_{2}] \ldots [g_{n},h_{n}],$$
where the $g_{i}$ and the $h_{i}$ lie in $G$. Then,
$$\left|q(g)- \sum_{i=1}^{n}(q(g_{i})+q(h_{i}g_{i}^{-1}h_{i}^{-1})) \right| \leq (2n-1) D(q)$$
and, as $q(h_{i}g_{i}^{-1}h_{i}^{-1})=-q(g_{i})$,
$$\left|q(g)\right| \leq 2 cl_{G}(g) D(q)$$
for all $g$ in $G$. Therefore, for every natural integer $p$,
$$\left|q(g^{p})\right|=p\left|q(g)\right| \leq 2 cl_{G}(g^{p}) D(q).$$
Hence, $cl_{G}(g^{p})\rightarrow +\infty$ when $p \rightarrow + \infty$ as soon as $q(g) \neq 0$.
\end{proof}

\bigskip

Quasi-morphisms on $G$ are in fact more closely related to commutator length via the Bavard duality (see \cite{Bav} and \cite{Ca}). A typical example is the \emph{translation number} on the group $\mathrm{Homeo}_{\mathbb{Z}}(\mathbb{R})$ of homeomorphisms of $\mathbb{R}$ which commute with integral translations (which is also the group of lifts of orientation-preserving homeomorphisms of the circle): for every homeomorphism $f$ in $\mathrm{Homeo}_{\mathbb{Z}}(\mathbb{R})$, the sequence $(\frac{f^{n}(x)-x}{n})_{n}$ converges for any $x$ and the limit is independant of the chosen point $x$. This limit is called the translation number of $f$. It actually defines a quasi-morphism on the group $\mathrm{Homeo}_{\mathbb{Z}}(\mathbb{R})$ (see \cite{Gh} for more information on the translation number).
\bigskip

We now consider the closed annulus $\mathbb{A}=\mathbb{R}/\mathbb{Z} \times [0,1]$. We build a quasi-morphism on the identity component $\mathrm{Homeo}_{0}(\mathbb{A})$ of the group of homeomorphisms of $\mathbb{A}$. The following construction is due to Frederic Le Roux. We denote by $\pi : \mathbb{R} \times[0,1] \rightarrow \mathbb{A}$ the universal covering of $\mathbb{A}$. Given a homeomorphism $f$ in $\mathrm{Homeo}_{0}(\mathbb{A})$, consider a lift $F:\mathbb{R} \times [0,1] \rightarrow \mathbb{R}\times [0,1]$ of $f$, \textit{i.e.} a homeomorphism of $\mathbb{R} \times [0,1]$ which satisfies $\pi \circ F=f \circ \pi$. The maps $F_{0}=F_{\mathbb{R} \times\left\{0\right\}}$ and $F_{1}=F_{\mathbb{R} \times \left\{1\right\}}$ belong to the group $\mathrm{Homeo}_{\mathbb{Z}}(\mathbb{R})$. If we replace $F$ by some other lift $F+k$, then both translation numbers of $F_{0}$ and $F_{1}$ will increase by $k$. Thus the difference of the translation numbers of these homeomorphisms is independant of the lift $F$ chosen. We denote this number by $\rho(f)$ and call it the \emph{torsion number}. As the translation number is a homogeneous quasi-morphism on the group $\mathrm{Homeo}_{\mathbb{Z}}(\mathbb{R})$, $\rho$ is a (nontrivial) homogeneous quasi-morphism on $\mathrm{Homeo}_{0}(\mathbb{A})$ which restricts to a non-trivial quasi-morphism on $\mathrm{Diff}_{0}^{r}(\mathbb{A})$ for every $r$ in $\mathbb{N} \cup \left\{ \infty \right\}$.

We first reduce Theorem \ref{qm} to the following proposition, which gives estimates on the commutator length. We note $E$ the lower integer part and, for an element $f$ in $\mathrm{Diff}_{0}^{r}(\mathbb{A})$:
$$\alpha(f)= \min_{r \in \mathbb{R}} \left|F_{1}(r)- F_{0}(r) \right|,$$ 
where $F$ is a lift of $f$ (note that this quantity is independant of the choice of the lift).

\begin{proposition} \label{cl}
Let $f \neq Id_{\mathbb{A}}$ be a homeomorphism in $\mathrm{Diff}_{0}^{r}(\mathbb{A})$ and $F$ be a lift of $f$. Then, for any $r \neq 2,3$: 
$$E \left( \frac{\alpha(f)+3}{4} \right) +9 \ \geq \ cl_{r}(f) \ \geq \ E \left( \frac{\alpha(f)}{4} \right) +1.$$
If $r=0$, then the $9$ appearing in the upper bound may be improved to $5$.
\end{proposition}

This proposition will be shown in the next section. Now, we deduce the first part of Theorem \ref{qi} and Theorem \ref{qm} from Proposition \ref{cl}. 

\begin{proof}[Proof of the first part of Theorem \ref{qi}]
It suffices to show that, for every element $f$ in $\mathrm{Diff}_{0}^{r}(\mathbb{A})$:
$$\left| |\rho(f)| - \alpha(f) \right| \leq 2.$$
Take a real number $r_{0}$ in $\mathbb{R}$ such that $\alpha(f)=\left| F_{1}(r_{0}) - F_{0}(r_{0}) \right|$, where $F$ is a lift of $f$.
We will prove that :
$$\left| \rho(f) - F_{1}(r_{0}) - F_{0}(r_{0}) \right| \leq 2.$$
For every homeomorphisms $G$ and $H$ in $\mathrm{Homeo}_{\mathbb{Z}}(\mathbb{R})$, we have the following classical inequality:
$$\forall r \in \mathbb{R}, \ \left| (G(H(r))-r)-(G(r)-r)-(H(r)-r) \right| \leq 1.$$
Using this formula, we obtain by induction, for $i\in \left\{0,1 \right\}$:
$$ \forall n \in \mathbb{N}, \ \left| F_{i}^{n}(r_{0})-nF_{i}(r_{0})+(n-1)r_{0} \right| \leq n-1$$
and $\left| F_{1}^{n}(r_{0})-F_{0}^{n}(r_{0})-n(F_{1}(r_{0})-F_{0}(r_{0})) \right| \leq 2n-2.$
Dividing by $n$ and taking the limit $n \rightarrow +\infty$ allows us to conclude the proof.
\end{proof}

\bigskip

\begin{proof}[Proof of Theorem \ref{qm}]
The strategy of the proof is the same as in the case of the translation number on $\mathrm{Homeo}_{\mathbb{Z}}(\mathbb{R})$ (see \cite{Gh}). Let $q$ be a homogeneous quasi-morphism on $\mathrm{Diff}_{0}^{r}(\mathbb{A})$. Let $t$ be a $C^{\infty}$-diffeomorphism in $\mathrm{Diff}_{0}^{\infty}(\mathbb{A})$ which is a "twist": there exists a lift $T$ of $t$ which satisfies:
$$\forall x \in \mathbb{R},\left\{
\begin{array}{l}
 T(x,0)=(x,0). \\
 T(x,1)=(x+1,1).
\end{array} 
\right.
$$
(For instance, one can take $t(x,r)=(x+r,r)$.)
Then $q-q(t)\rho$ is a homogeneous quasi-morphism which vanishes on $t$ and the next lemma allows us to complete the proof of Theorem \ref{qm}.
\end{proof}

\underline{Remark}: observe that $t$ fixes the boundary but there is no continuous path of diffeomorphisms fixing the boundary between the identity and $t$.

\begin{lemma} \label{t}
Every homogeneous quasi-morphism on $\mathrm{Diff}_{0}^{r}(\mathbb{A})$ which vanishes on $t$ is trivial, if $r \neq 2,3$.
\end{lemma}

\begin{proof}[Proof of lemma \ref{t}]
Let $q$ be such a quasi-morphism. Then, for every integer $n \in \mathbb{Z}$ and every homeomorphism $f$ in $\mathrm{Diff}_{0}^{r}(\mathbb{A})$, we have
$$\left|q(t^{n}f)-q(f) \right| \leq D(q).$$
Let us fix $f$ in $\mathrm{Diff}_{0}^{r}(\mathbb{A})$  and choose an integer $n_{0}$ such that $\alpha(t^{n_{0}}f)<1$ (if $\alpha(f)= \min_{r \in \mathbb{R}} (F_{1}(r)- F_{0}(r))$, then $\alpha(t^{p}f)=\alpha(f)+p$ as long as $p>-\alpha(f)$ and the other cases are similar or easier). Then, by proposition \ref{cl}, the homeomorphism $t^{n_{0}}f$ may be written as a product of at most $9$ commutators. Therefore, using lemma 1.2:
$$\left|q(t^{n_{0}}f)\right| \leq 18D(q)$$
and $\left|q(f) \right| \leq 19D(q)$. Hence, $q$ is a bounded homogeneous quasimorphism: it is trivial. 
\end{proof}

\section{Estimation of the commutator length}

This section is devoted to the proof of Proposition \ref{cl}. We will first make the proof in the case $r=0$ and then an approximation argument will give the $r>0$ case.
We will need a theorem by Eisenbud, Hirsch and Neumann on the commutator length for lifts of circle homeomorphisms.

\begin{theorem}(Eisenbud-Hirsch-Neumann (see \cite{EHN} Theorem 2.3))

A homeomorphism $F$ in $\mathrm{Homeo}_{\mathbb{Z}}(\mathbb{R})$ can be written as a product of $n$ commutators in $\mathrm{Homeo}_{\mathbb{Z}}(\mathbb{R})$ if and only if:
$$\min_{x \in \mathbb{R}} |F(x)-x| <2n-1.$$

\end{theorem}

\subsection{Lower bound of the commutator length for $r=0$}

Let $f$ be a homeomorphism in $\mathrm{Homeo}_{0}(\mathbb{A})$ with $cl_{0}(f)=n$. Let $F$ denote a lift of $f$ which can be written as a product of $n$ commutators in the group of lifts of elements in $\mathrm{Homeo}_{0}(\mathbb{A})$ namely
$$\left\{G \in \mathrm{Homeo}_{0}(\mathbb{R}\times [0,1]), \forall(x,t) \in \mathbb{R} \times [0,1], \ G(x+1,t)=G(x,t)+(1,0) \right\}.$$
Then $F_{0}$ and $F_{1}$, which belong to $\mathrm{Homeo}_{\mathbb{Z}}(\mathbb{R})$, can be written as a product of $n$ commutators. Using the Eisenbud-Hirsch-Neumann theorem, we get:
$$\min_{x \in \mathbb{R}} |F_{0}(x)-x| =|F_{0}(x_{0})-x_{0}| <2n-1$$
and:
$$\min_{x \in \mathbb{R}} |F_{1}(x)-x|=|F_{1}(x_{1})-x_{1}| <2n-1.$$
We may assume that: $x_{0}<x_{1}<x_{0}+1$. These inequalities imply that:
$$\begin{array}{rcl}
 \alpha(f) & \leq & |F_{0}(x_{0})- F_{1}(x_{0}) | \\
 & \leq & |F_{0}(x_{0})-x_{0}| + |(F_{1}(x_{0})-x_{0})- (F_{1}(x_{1})-x_{1})|+|F_{1}(x_{1})-x_{1}| \\
 & < & (2n-1) +2+(2n-1) \\
 & < & 4n
\end{array}
$$
The lower bound in Proposition \ref{cl} is therefore proved.

\subsection{Upper bound of the commutator length for $r=0$}

To get an upper bound, we first compose a given homeomorphism by some number of commutators to get a homeomorphism which admits a lift which pointwise fixes the boundary. Then we compose by 2 commutators to get a homeomorphism with a lift which is the identity in a neighbourhood of the boundary. Finally, such a homeomorphism can be written as a product of 2 commutators, by a result by Burago, Ivanov and Polterovich.

Let $f$ be a homeomorphism in $\mathrm{Homeo}_{0}(\mathbb{A})$. Let $k=E(\frac{\alpha(f)+3}{4})+1$.

\subsubsection{First step: getting a homeomorphism with a lift which fixes pointwise the boundary by composing with $k$ commutators}

Up to conjugating by $(x,r) \rightarrow (x,1-r)$, we may assume that $f$ is a "positive twist": there exists a lift $F$ of $f$ and a real number $x_{0}$ which satisfy $\alpha(f)=F_{1}(x_{0})-F_{0}(x_{0})$.

By composing the lift $F$ by an integral translation if necessary, we may suppose that:
$$2k-2 \leq F_{1}(x_{0})-x_{0} < 2k-1.$$
Then, as $4k>\alpha(f)+3$:
$$ F_{0}(x_{0})-x_{0}=F_{1}(x_{0})-x_{0}-\alpha(f)\geq 2k-2-\alpha(f) >1-2k. $$
By using the Eisenbud-Hirsch-Neumann theorem, we get that $F_{1}$ and $F_{0}$ can be written as products of $k$ commutators in $\mathrm{Homeo}_{\mathbb{Z}}(\mathbb{R})$:
$$\left\{
\begin{array}{l}
F_{1}=\prod_{i=1}^{k}[g_{i,1},h_{i,1}] \\
F_{0}=\prod_{i=1}^{k}[g_{i,0},h_{i,0}]
\end{array}
\right. ,
$$
where the $g_{i,j}$'s and the $h_{i,j}$'s belong to $\mathrm{Homeo}_{\mathbb{Z}}(\mathbb{R})$.

For every index $i$, let us take $G_{i} :\mathbb{R} \times [0,1] \rightarrow \mathbb{R} \times [0,1]$ (respectively $H_{i}$) a lift of an element $g_{i}$ (respectively $h_{i}$) of $\mathrm{Homeo}_{0}(\mathbb{A})$ which satisfies:
$$\left\{
\begin{array}{l}
p_{1} \circ G_{i}(.,1)=g_{i,1} \\
p_{1} \circ G_{i}(.,0)=g_{i,0} 
\end{array}
\right. 
$$
(respectively:
$$\left\{
\begin{array}{l}
p_{1} \circ H_{i}(.,1)=h_{i,1} \\
p_{1} \circ H_{i}(.,0)=h_{i,0} 
\end{array}
\right. ).
$$
Note that the homeomorphisms $H_{i}$ and $G_{i}$ exist because the $g_{i,j}$'sand the $h_{i,j}$'s are isotopic to the identity.
Then the homeomorphism $[h_{k},g_{k}] \circ [h_{k-1},g_{k-1}] \circ \ldots \circ [h_{1},g_{1}] \circ f$ admits a lift which pointwise fixes the boundary.

\subsubsection{Second step: a homeomorphism with a lift which pointwise fixes the boundary can be written as a product of $4$ commutators}

Fix a homeomorphism $f$ in $\mathrm{Homeo}_{0}(\mathbb{A})$ with a lift $F$ which is the identity on the boundary. 

Consider a path $c:[0,1]\rightarrow \mathbb{A}$ which satisfies the following conditions, where $p_{1}$ and $p_{2}$ are respectively the first and the second projection of $\mathbb{A}=\mathbb{R} / \mathbb{Z} \times [0,1]$:
\begin{itemize}
\item $c(0)=(-\frac{1}{4},0)$ and $c(1)=(\frac{1}{4},0)$.
\item $\forall t \in (0,1), \ 0 < p_{2} \circ c(t) <\frac{1}{4}$.
\item $\forall t \in [0,1], \ p_{2} \circ F \circ c(t)< \frac{1}{4}$.
\item $\forall t \in [0,1], \ p_{1} \circ F \circ c(t) \in (-\frac{3}{8},\frac{3}{8})$.
\item $\forall t \in [0,1], \ p_{1} \circ c(t) \in (-\frac{3}{8},\frac{3}{8})$
\end{itemize}

Let $D$ be the connected component of $\mathbb{A}-(|-\frac{1}{4},\frac{1}{4}]\times \left\{ 0 \right\} \cup c([0,1]))$ whose boundary contains the point $(0,0)$. According to the Schönfliess theorem, there exists a homeomorphism $f_{1,0}$ in $\mathrm{Homeo}_{0}(\mathbb{A})$ which coincides with $f$ on the closure of $D$ (which is a closed topological disc) and with support included in $(-\frac{3}{8},\frac{3}{8})\times [0, \frac{1}{2})=D_{1,0}$. By taking a lace homotopic to the lower boundary component and close to it and applying the Schönfliess theorem, we can find a homeomorphism $f_{2,0}$ which coincides with $f_{1,0}^{-1} \circ f$ (which is the identity on $D$) in a neighbourhood of the lower boundary $\mathbb{R} \times \left\{ 0 \right\}$ and which satisfies:
$$supp(f_{2,0}) \subset (\frac{1}{8},\frac{7}{8}) \times  [0,\frac{1}{2})=D_{2,0}.$$

With the same technique, we can build  homeomorphisms $f_{1,1}$ and $f_{2,1}$ in $\mathrm{Homeo}_{0}(\mathbb{A})$ which satisfy:
\begin{itemize}
\item the homeomorphism $f_{1,1} \circ f_{2,1}$ coincides with $f$ in a neighbourhood of the upper boundary $\mathbb{R} \times \left\{ 1 \right\}$.
\item the support of $f_{1,1}$ is included in $(-\frac{3}{8},\frac{3}{8}) \times  (\frac{1}{2},1]=D_{1,1}$.
\item the support of $f_{2,1}$ is included in $(\frac{1}{8},\frac{7}{8}) \times  (\frac{1}{2},1]=D_{2,1}$.
\end{itemize}

Thus, we have $f=f_{1,1} \circ f_{2,1} \circ f_{1,0} \circ f_{2,0}$ near the boundary.

\begin{lemma} \label{halfplane} Let $\mathbb{H}$ denote the closed upper half plane. Then every element in $\mathrm{Homeo}_{0}(\mathbb{H})$ can be written as one commutator.
\end{lemma}

\begin{proof} 
Let $h$ be an element in $\mathrm{Homeo}_{0}(\mathbb{H})$.
Let $U$ be a neighbourhood of $supp(h)$ and $\varphi$ be a homeomorphism in $\mathrm{Homeo}_{0}(\mathbb{H})$ which satisfies the following conditions:\\
- the open sets $\varphi^{k}(U)$, for $k$ in $\mathbb{N}$, are pairwise disjoint.\\
- the sequence $(\varphi^{k}(U))_{k \in \mathbb{N}}$ converges to a singleton $\left\{p \right\}$ which lies on the boundary as $k$ tends to $+\infty$.

Now, consider the homeomorphism $g$ in $\mathrm{Homeo}_{0}(\mathbb{H})$ which satisfies:\\
- $g=Id$ outside $\bigcup_{k \in \mathbb{N}} \varphi^{k}(U)$.\\
- for every non-negative integer $k$, $g_{| \varphi^{k}(U)}=\varphi^{k} \circ h \circ \varphi^{-k}$.

Then $h=[g, \varphi]$.
\end{proof}

According to the lemma applied in the discs $D_{i,j}$, each $f_{i,j}$ is a commutator, thus $f$ coincides with a product of $4$ commutators in a neighbourhood of the boundary. We will see that we can improve this to $2$.
Using this last lemma, we may consider homeomorphisms $g_{i,j}$ and $h_{i,j}$ supported in $D_{i,j}$ which satisfy $f_{i,j}=[g_{i,j},h_{i,j}]$. Note that:
$$supp(g_{i,0}) \cap supp(g_{i,1})= \emptyset, $$
$$supp(h_{i,0}) \cap supp(h_{i,1})= \emptyset,$$
$$supp(g_{i,0}) \cap supp(h_{i,1})= \emptyset,$$
$$supp(h_{i,0}) \cap supp(g_{i,1})= \emptyset,$$
and
$$supp(f_{i,0}) \cap supp(f_{i,1})= \emptyset.$$
Thus, those pairs of homeomorphisms commute. Therefore:
$$\begin{array}{rcl}
g=f_{1,1}\circ f_{2,1} \circ f_{1,0} \circ f_{2,0} & = & f_{1,1} \circ f_{1,0} \circ f_{2,1} \circ f_{2,0} \\
 & = & [g_{1,1},h_{1,1}] \circ [g_{1,0},h_{1,0}] \circ [g_{2,1},h_{2,1}] \circ [g_{2,0},h_{2,0}]\\
 & = & [g_{1,1} \circ g_{1,0}, h_{1,1} \circ h_{1,0}] \circ [g_{2,1}\circ g_{2,0},h_{2,1} \circ h_{2,0}]
\end{array}
$$
Moreover, this homeomorphism coincides with $f$ on a neighbourhood of the boundary. Thus $g^{-1} \circ f$ admits a lift which is the identity near the boundary. As the open annulus is a portable manifold (see \cite{BIP} Theorem 1.18), the homeomorphism $g^{-1} \circ f$ can be written as a product of two commutators.
This ends the proof of Proposition \ref{cl} in the case $r=0$. 

\subsection{Case $r>0$}

It follows directly from the result on homeomorphisms that:

$$\forall f \in \mathrm{Diff}_{0}^{r}(\mathbb{A}), \ cl_{r}(f)\geq cl_{0}(f) \geq E(\frac{\alpha(f)}{4})+1.$$

To get an upper bound (for $r \neq 2,3$), the process is the following. We first write a diffeomorphism $f$ as a product of $E(\frac{\alpha(f)+3}{4})+5$ commutators of homeomorphisms and we approximate every homeomorphism appearing in this product by a diffeomorphism. Hence, arbitrarily close to $f$ in the $C^{0}$ topology, there is a product of $E(\frac{\alpha(f)}{4})+5$ commutators of diffeomorphisms. To conclude, it suffices to notice that, for $r \neq 2,3$, a diffeomorphism sufficiently close to the identity can be written as a product of $4$ commutators. Let us show this last fact. We need the following lemma, which is a consequence of a fragmentation lemma proved in the appendix:

\begin{lemma}\label{fragannulus}
There exists a $C^{0}$ neighbourhood $\Upsilon$  of the identity in $\mathrm{Diff}_{0}^{r}(\mathbb{A})$ such that:
$$\forall f \in \Upsilon, \ \exists f_{1}, f_{2} \in \mathrm{Diff}_{0}^{r}(\mathbb{A}), 
\left\{
\begin{array}{l}
f=f_{1} \circ f_{2} \\
supp(f_{1}) \subset (-\frac{3}{8},\frac{3}{8}) \times [0,1] \\
supp(f_{2}) \subset (\frac{1}{8},\frac{7}{8}) \times [0,1]
\end{array}
\right.
.
$$
\end{lemma}

Now, the manifolds $(-\frac{3}{8},\frac{3}{8}) \times [0,1]$ and $(\frac{1}{8},\frac{7}{8}) \times [0,1]$ are portable, in the sense of \cite{BIP}. It follows from \cite{BIP} Theorem 1.18 that, if $r \neq 2, 3$, the diffeomorphisms $f_{1}$ and $f_{2}$ may be each written as a product of two commutators.

This concludes the proof in the case $r>0$.

\section{Estimation of the fragmentation norm}

The analog of Proposition \ref{cl} for the fragmentation norm is the following proposition, which implies the second part of Theorem \ref{qi}:

\begin{proposition} \label{frag}
Let $f$ be a homeomorphism in $\mathrm{Diff}_{0}^{r}(\mathbb{A})$ and $F:\mathbb{R} \times [0,1] \rightarrow \mathbb{R} \times [0,1]$ be a lift of $f$. Suppose $f \neq Id$. Then, for any $r$ in $\mathbb{N} \cup \left\{\infty\right\}$: 
$$E(\alpha(f))+40 \ \geq \ \mathrm{Frag}_{r}(f) \ \geq \ E(\alpha(f))+1.$$
If $r=0$, then the $40$ appearing in the upper bound may be improved to $36$.
\end{proposition}

Actually, we can have far better upper bound in the preceding proposition by using the following result which will be showed in a next paper:

\begin{proposition} \label{openannulus}
For any $r \in \mathbb{N} \cup \left\{\infty\right\}$, the fragmentation norm on $\mathrm{Diff}_{0}^{r}(int(\mathbb{A}))$ is bounded by $4$.
\end{proposition}

If we admit this last proposition, we get more precise estimates on the fragmentation norm:
\begin{proposition} \label{frag+}
Let $f$ be a homeomorphism in $\mathrm{Diff}_{0}^{r}(\mathbb{A})$ and $F:\mathbb{R} \times [0,1] \rightarrow \mathbb{R} \times [0,1]$ be a lift of $f$. Suppose $f \neq Id$. Then, for any $r$ in $\mathbb{N} \cup \left\{\infty\right\}$: 
$$E(\alpha(f))+16 \ \geq \ \mathrm{Frag}_{r}(f) \ \geq \ E(\alpha(f))+1.$$
If $r=0$, then the $16$ appearing in the upper bound may be improved to $12$.
\end{proposition}

To make estimates on the fragmentation norm for the closed annulus, we need an analog of the Eisenbud-Hirsch-Neumann theorem.

\subsection{Fragmentation norm on $\mathrm{Homeo}_{\mathbb{Z}}(\mathbb{R})$}

Let $A$ be the subset of $\mathrm{Homeo}_{\mathbb{Z}}(\mathbb{R})$ given by elements which fix pointwise a nonempty open interval. This subset generates $\mathrm{Homeo}_{\mathbb{Z}}(\mathbb{R})$ as a group. If $F$ is a homeomorphism in $\mathrm{Homeo}_{\mathbb{Z}}(\mathbb{R})$, we denote $\mathrm{Frag}(F)$ the minimal number of elements of $A$ necessary to write $F$ as a product of elements of $A$. The analog of the Eisenbud-Hirsch-Neumann theorem for the fragmentation norm is the following:

\begin{proposition} \label{fragcircle}
Let $F$ be a homeomorphism in $\mathrm{Homeo}_{\mathbb{Z}}(\mathbb{R})$. Then, for any $k \geq 0$:
$$\mathrm{Frag}(F) \leq k+2 \Longleftrightarrow \min_{x \in \mathbb{R}} |F(x)-x| < k+1.$$
\end{proposition}

\begin{proof}
Let us start with the direct implication. Write:
$$ F=F_{1}F_{2} \ldots F_{k+1}F_{k+2},$$
where each $F_{i}$ belongs to $A$. Take a point $x_{0}$ such that $F_{k+2}(x_{0})=x_{0}$. We have, for every integer $i \in [1,k+1]$:
$$|F_{i}F_{i+1} \ldots F_{k+1}F_{k+2}(x_{0})-F_{i+1} \ldots F_{k+1}F_{k+2}(x_{0})| < 1$$
and by summing these inequalities:
$$|F(x_{0})-x_{0}|<k+1.$$
This proves the direct implication.
Let us show the converse by induction on $k$.
For $k=0$, up to conjugating by $(x,r) \rightarrow (x,1-r)$, we may assume:
$$0\leq \min_{x \in \mathbb{R}} |F(x)-x|=F(x_{0})-x_{0} < 1.$$
Choose a point $x_{1}$ such that $F(x_{0})<x_{1}<x_{0}+1$. Then $F(x_{0})<F(x_{1})<F(x_{0})+1$ and we can find a homeomorphism $h$ in $\mathrm{Homeo}_{\mathbb{Z}}(\mathbb{R})$ which fixes a neighbourhood of $F(x_{0})$ and satisfies $h\circ F=Id$ in a neighbourhood of $x_{1}$. The decomposition $F=h^{-1} \circ (h \circ F)$ shows that $\mathrm{Frag}(F) \leq 2$.

Suppose the converse holds for an integer $k$. Let us prove it for the integer $k+1$. Suppose $\min_{x \in \mathbb{R}} |F(x)-x| < k+2$. We may also suppose that $\min_{x \in \mathbb{R}} |F(x)-x| \geq k+1$ (otherwise, we can conclude directly from the induction hypothesis). As usual, we may assume:
$$k+1 \leq \min_{x \in \mathbb{R}} (F(x)-x) < k+2.$$
Let $x_{0}$ be a point which satisfies:
$$\min_{x \in \mathbb{R}}(F(x)-x)=F(x_{0})-x_{0}.$$
The same way as for the initialization, we can find a point $x_{1} \in (x_{0},x_{0}+1)$ such that $F(x_{0})<x_{1}+k+1<F(x_{0})+1$. Then $F(x_{0})< F(x_{1}) < F(x_{0})+1$. Therefore, there exists a homeomorphism $h$ which fixes a neighbourhood of $F(x_{0})$ and which satisfies:
$$h(F(x_{1}))<x_{1}+k+1.$$
The induction hypothesis allows then us to finish the induction and the proof of Proposition \ref{fragcircle}.
\end{proof}

\subsection{Proof of Proposition \ref{frag}}

The technique is exactly the same as in the proof of Proposition \ref{cl}, using Proposition \ref{fragcircle} instead of the Eisenbud-Hirsch-Neumann theorem. Thus the explanations will be briefer.

\subsubsection{Lower bound for the fragmentation norm}

Let $f$ be an element in $\mathrm{Homeo}_{0}(\mathbb{A})$ with $\mathrm{Frag}_{0}(f) \geq 1$ and $F$ be a lift of $f$. Fix a decomposition of $f$ as a product of $\mathrm{Frag}_{0}(f)$ homeomorphisms supported in discs. Let $k_{0}$ (respectively $k_{1}$) be the number of homeomorphisms appearing in this decomposition whose support meets the lower boundary $\mathbb{R} / \mathbb{Z} \times \left\{ 0 \right\}$ (resp. the upper boundary $\mathbb{R} / \mathbb{Z} \times \left\{ 1 \right\}$) of $\mathbb{A}$. Note that $k_{0}+k_{1} \leq \mathrm{Frag}_{0}(f)$ as a disc doesn't touch both components of the boundary, by definition. First, suppose $k_{0}\geq 2$ and $k_{1}\geq 2$. By Proposition \ref{fragcircle}, there exist points $x_{0}$ and $x_{1}$ in $\mathbb{R}$ such that:
$$|F_{0}(x_{0})-x_{0}|< k_{0}-1$$
and:
$$|F_{1}(x_{1})-x_{1}|< k_{1}-1.$$
From these inequalities, we get:
$$\alpha(f) < k_{0}-1+2+k_{1}-1\leq \mathrm{Frag}_{0}(f),$$
which gives the lower bound in Proposition \ref{frag}.
In the cases $k_{0}=1$ and $k_{1} \geq 2$ or $k_{1}=1$ and $k_{0} \geq 2$, one of the inequalities
$$|F_{0}(x_{0})-x_{0}| \leq k_{0}-1$$
and
$$|F_{1}(x_{1})-x_{1}| \leq k_{1}-1$$
is indeed strict, which allows us to conclude.

In the case $k_{0}=0$ and $k_{1} \geq 2$  (the case $k_{1}=0$ and $k_{0} \geq 2$ is symmetric), we have, as $F_{0}=Id$:
$$\alpha(f) < |F_{1}(x_{1})-x_{1}| +1 \leq k_{1} \leq \mathrm{Frag}_{0}(f).$$

In the case $(k_{0},k_{1})=(1,1)$, we have $\alpha(f) < 1$ and $\mathrm{Frag}_{0}(f)\geq 2$ so the inequality $\alpha(f)<\mathrm{Frag}_{0}(f)$ holds.

In the cases $(k_{0},k_{1}) \in \left\{(0,0),(0,1),(1,0) \right\}$, we have $\alpha(f)=0$ so the inequality $\alpha(f)<\mathrm{Frag}_{0}(f)$ holds. 

\subsubsection{Upper bound for the fragmentation norm}

Let $f$ be an element in $\mathrm{Homeo}_{0}(\mathbb{A})$ with lift $F$. As usual, we may assume: $\alpha(f)=p_{1} \circ F(x_{0},1)-p_{1} \circ F(x_{0},0)$. Consider a lift of $f$ which satisfies:
$$-1 < F_{0}(x_{0})-x_{0} \leq 0.$$
Then $F_{1}(x_{0})=F_{0}(x_{0})+\alpha(f)< E(\alpha(f))+1.$
Using Proposition \ref{fragcircle}, we can see that, after composing by at most $E(\alpha(f))+2$ homeomorphisms supported in discs which touch the upper boundary, we get a homeomorphism with a lift which fixes the upper boundary and, by composing by two more homeomorphisms supported in discs we get a homeomorphism with a lift which fixes both boundary components. Then, by composing by four homeomorphisms supported in discs, we get a homeomorphism with a lift which fixes a neighbourhood of the boundary. Finally, by the result by Burago, Ivanov and Polterovich (see the proof of Theorem 1.17 in \cite{BIP}), such a homeomorphism may be written as a product of $28$ homeomorphisms supported in discs. Using proposition \ref{openannulus}, we see that this might be improved to $4$.

Then, an approximation argument combined with the fragmentation lemma yields the case $r>0$.

\section{Generalization to other surfaces}

A similar construction as the one made in section 2 can be carried out on every compact surface $M$ with boundary to obtain quasi-morphisms. However, in those cases, we do not know the dimension of the space of quasi-morphism: we just have a minoration of it.

\subsection{Case of open surfaces}

Suppose $M$ is a non-compact surface with $p$ boundary components which are circles. Let us fix such a boundary component of $M$: $C_{1}$. Take a path homeomorphic to the half-line which begins on $C_{1}$, goes to infinity, and touches the boundary component $C_{1}$ only at endpoint. Consider the cyclic covering $p:\tilde{M} \rightarrow M$ associated to this path. Now, a homeomorphism $f$ in $\mathrm{Homeo}_{0}(M)$ admits a unique lift $F:\tilde{M}\rightarrow \tilde{M}$ which is compactly supported. The translation number of the restriction of $F$ to $p^{-1}(C_{1})$ gives rise to a quasi-morphism. With this method, $p$ independant quasi-morphisms can be built. Moreover, if the commutator length is bounded on the group $\mathrm{Diff}_{0}^{r}(\mathrm{int}(M))$, where $\mathrm{int}(M)$ denotes $M- \partial M$, then it can be shown with the same techniques as in the case of the annulus that the vector space of homogeneous quasi-morphisms is p-dimensional. In particular, the vector space of homogeneous quasi-morphism is one-dimensional in the case of the half-open annulus $\mathbb{R}/\mathbb{Z} \times [0,1)$. However, it is not known whether the commutator length is bounded or not on the group $\mathrm{Diff}_{0}^{r}(\mathrm{int}(M))$ when $\mathrm{int}(M)$ is different from the open annulus or from $\mathbb{R}^{2}$. 

\subsection{Case of other compact oriented surfaces with boundary}

When $M$ is a compact oriented surface with boundary which is different from the closed disc or the closed annulus, its universal covering $\tilde{M}$ may be seen as a subspace of the Poincaré disc $\mathbb{D}$. We endow $M$ and $\mathbb{D}$ with riemmannian metrics of constant curvature $-1$ such that $M$ has a geodesic boundary and the projection $\tilde{M} \rightarrow M$ is a riemannian covering. Denote by $d$ the associated distance. Given two points $x$ and $y$ on $\mathbb{D}$ and an oriented geodesic $g$ which contains both $x$ and $y$, we define $a(x,y,g)=d(x,y)$ if the geodesic from $x$ to $y$ has the same orientation as $g$ and $a(x,y,g)=-d(x,y)$ otherwise.

Let us take a homeomorphism $f$ in $\mathrm{Homeo}_{0}(M)$. The homeomorphism $f$ admits a canonical lift $F:\tilde{M} \rightarrow \tilde{M}$ which is the identity on the limit set of $\tilde{M}$ at infinity. Given a point $x$ on the boundary component $C$ of $\tilde{M}$, which is a geodesic, define:
$$q(x)=\lim_{n \rightarrow + \infty} \left( \frac{a(x,F^{n}(x),C)}{n} \right).$$
This number is the translation number of $F$ restricted to $C$ and does not depend on the point $x$ chosen on $C$. With this construction, we can build $p$ independant quasi-morphisms, where $p$ is equal to the number of boundary component of $M$. Here also, the vector space of quasi-morphisms on $\mathrm{Diff}_{0}^{r}(M)$ is $p$-dimensional if the commutator length on $\mathrm{Diff}^{r}_{0}(\mathrm{int}(M))$ is bounded.

\section*{Acknowledgements}

I would like to thank Frederic Le Roux for suggesting me to write this paper and for his careful reading of it.

\appendix
\section{Appendix: A $C^{0}$-fragmentation lemma for diffeomorphisms}

In this section, we will prove lemma \ref{fragannulus}. This lemma follows directly from the following proposition which is more general:

\begin{proposition}\label{fraglemma}
Let $M$ be a compact surface, possibly with boundary and $(U_{i})_{1 \leq i \leq n}$ be a finite open covering of $M$. Denote by $d$ a distance on $\mathrm{Homeo}_{0}(M)$ compatible with the $C^{0}$-topology. Then, for every $\epsilon>0$, there exists $\alpha>0$ such that, given a diffeomorphism $f$ in $\mathrm{Diff}^{r}_{0}(M)$ satisfying $d(f,Id_{M})<\alpha$, there exist diffeomorphisms $f_{1}, f_{2}, \ldots, f_{n}$ such that:
\begin{itemize}
\item $f=f_{1} \circ f_{2} \circ \ldots \circ f_{n}$.
\item for every index $i$, $supp(f_{i}) \subset U_{i}$.
\item $d(f_{i},Id_{M})<\epsilon$.
\end{itemize}
\end{proposition}

To prove this proposition, we will need the following extension lemma.

The half-closed disk is the subset of $\mathbb{R}^{2}$:
$$\left\{(x,y)\in \mathbb{R}^{2}, \left\{ \begin{array}{l} x\geq 0 \\ x^{2}+y^{2} \leq 1 \end{array} \right. \right\} .$$
\begin{lemma}\label{extensionlemma}
Let $U$ be an open set, $F$ be a closed set in $M$ and $N$ be an open neighbourhood of $F$. Let $K$ be a subset of $U$ which is $C^{\infty}$-diffeomorphic to a closed disc if $K\cap \partial M=\emptyset$ or to a half-closed disk if $K\cap \partial M \neq \emptyset$. Then, for every $\epsilon>0$, there exists $\beta>0$ such that, given a diffeomorphism $f$ in $\mathrm{Diff}_{0}^{r}(M)$ satisfying $d(f,Id_{M})<\beta$ and $f=Id_{M}$ on $N$, there exists a diffeomorphism $g$ in $\mathrm{Diff}_{0}^{r}(M)$ such that:
\begin{itemize}
\item $supp(g) \subset U$.
\item $f=g$ in a neighbourhood of $K$.
\item $g=Id_{M}$ in a neighbourhood of $F$.
\item $d(g,Id_{M})<\epsilon$.
\end{itemize}
\end{lemma}

Let us see first how this lemma implies Proposition \ref{fraglemma}.

\begin{proof}[Proof of Proposition \ref{fraglemma}]
First, we claim that it suffices to prove the proposition when each open set $U_{i}$ is homeomorphic to an open disc. Indeed, suppose the proposition holds for a covering by open disks. Take an arbitrary finite covering $(U_{i})_{1\leq i \leq n}$. Then we can find a covering by open disks $D_{i,j}$ with $1\leq i \leq n$ and $1\leq j \leq k_{i}$ such that $D_{i,j}$ is included in $U_{i}$. Then, if $f$ is a diffeomorphism of $M$ sufficiently close to the identity, there exist diffeomorphisms $f_{i,j}$ close to the identity such that:
$$\left\{
\begin{array}{l}
supp(f_{i,j})\subset D_{i,j} \\
f=f_{1,1} \circ f_{1,2} \circ \ldots \circ f_{1,k_{1}} \circ f_{2,1} \circ \ldots \circ f_{n,k_{n}}
\end{array}
\right.
.$$
It suffices to take $f_{i}=f_{i,1} \circ f_{i,2} \circ \ldots \circ f_{i,k_{i}}$to conclude the proof for an arbitrary finite covering.

Suppose now that each open set $U_{i}$ is homeomorphic to an open disc. For each index $i$, take a subset $K_{i}$ of $U_{i}$ diffeomorphic to a closed disc in such a way that:
$$\bigcup_{1 \leq i \leq n} K_{i}=M.$$
We show by induction on $i$ the following statement  (the case $i=n$ proves Lemma \ref{fraglemma}): for every $\epsilon>0$, there exists $\alpha>0$ such that, given a diffeomorphism $f$ in $\mathrm{Diff}_{0}^{r}(M)$ satisfying $d(f,Id_{M})<\alpha$, there exist diffeomorphisms $f_{1}, f_{2}, \ldots, f_{i}$ in $\mathrm{Diff}_{0}^{r}(M)$ such that:
\begin{enumerate}
\item $f=f_{1} \circ f_{2} \circ \ldots \circ f_{i}$ in a neighbourhood of $\bigcup_{j\leq i} K_{j}$.
\item $\forall j\leq i, supp(f_{j}) \subset U_{j}$.
\item $d(f_{j},Id_{M})<\epsilon$.
\end{enumerate}
The above statement is true for $i=1$ by lemma \ref{extensionlemma}.
Suppose the above statement holds for an integer $i$. Fix $\epsilon>0$. Let $\beta$ be given by Lemma \ref{extensionlemma} applied with $F=\bigcup_{j \leq i}K_{j}$, $K=K_{i+1}$ and $U=U_{i+1}$. Using the induction hypothesis for small enough $\epsilon'$, there exists $\alpha>0$ such that if we take a diffeomorphism $f$ $\alpha$-close to the identity, we can get a family of diffeomorphisms $(f_{j})_{1\leq j \leq i}$ which satisfies 1., 2., and 3. and such that $f_{i}^{-1} \circ f_{i-1}^{-1} \circ \ldots \circ f_{1}^{-1} \circ f$ is $\beta$-close to the identity. By Lemma \ref{extensionlemma}, we can find a diffeomorphism $f_{i+1}$ in $\mathrm{Diff}_{0}^{r}(M)$ such that :
\begin{enumerate}
\item $supp(f_{i+1}) \subset U_{i+1}$.
\item $f_{i+1}=f_{i}^{-1} \circ f_{i-1}^{-1} \circ \ldots \circ f_{1}^{-1} \circ f$ in a neighbourhood of $K_{i+1}$.
\item $f_{i+1}=Id_{M}$ in a neighbourhood of $\bigcup_{j\leq i}K_{j}$.
\item $d(f_{i+1},Id_{M})<\epsilon$.
\end{enumerate}
The properties 2. and 3., together with the induction hypothesis implie then that $f=f_{1} \circ f_{2} \circ \ldots \circ f_{i+1}$ on a neighbourhood of $\bigcup_{j \leq i}K_{j}$. This concludes the proof.
\end{proof}

To prove Lemma \ref{extensionlemma}, we need some specific extension lemmas. The first lemma is proved in the appendix of \cite{EPP} by M. Khanevsky.

\begin{lemma} \label{extembcirc}
Let $\epsilon>0$, $\epsilon_{1}\in (0,1)$, $r\in \mathbb{N} \cup \left\{\infty\right\}$. Consider a finite union of closed intervals $I \subset \mathbb{S}^{1}$. Then there exists $\alpha>0$ such that, for every $C^{r}$-embedding $f:\mathbb{S}^{1}\times(-\epsilon_{1},\epsilon_{1}) \rightarrow \mathbb{S}^{1} \times (-1,1)$ which satisfies:
$$\left\{
\begin{array}{l}
d(f,Id)<\alpha \\
f_{|I\times(-\epsilon_{1},\epsilon_{1})}=Id
\end{array}
\right. ,$$
there exists a diffeomorphism $g$ in $\mathrm{Diff}_{0}^{r}(\mathbb{S}^{1} \times (-1,1))$ such that:
\begin{enumerate}
\item $d(g,Id)<\epsilon$
\item $g_{|\mathbb{S}^{1} \times {0}}=f_{|\mathbb{S}^{1} \times {0}}$
\item $g_{|I\times(-1,1)}=Id$.
\end{enumerate}
\end{lemma}

The following lemma deals with embeddings of the annulus which preserve a circle.

\begin{lemma} \label{extembann}
Let $\epsilon_{1}\in [0,1)$, $r\in \mathbb{N} \cup \left\{\infty\right\}$ and $I$ denote a finite union of closed intervals in $\mathbb{S}^{1}$. Denote by $f$ a $C^{r}$-embedding of $\mathbb{S}^{1}\times (-\epsilon_{1},\epsilon_{1})$ in $\mathbb{S}^{1}\times (-1,1)$ which fixes pointwise $\mathbb{S}^{1} \times \left\{0\right\}$ and $I\times (-\epsilon_{1},\epsilon_{1})$. Then, for any $\epsilon>0$, there exists a diffeomorphism $g$ in $\mathrm{Diff}_{0}^{r}(\mathbb{S}^{1} \times (-1,1))$ fixing pointwise $\mathbb{S}^{1} \times \left\{0\right\}$ which satisfies:
\begin{enumerate}
\item $g=f$ in a neighbourhood of $\mathbb{S}^{1}\times\left\{0\right\}$.
\item $d(g,Id)<\epsilon$
\item $g_{|I\times(-1,1)}=Id$.
\end{enumerate}
\end{lemma}

\begin{proof}
Consider $\epsilon'>0$ such that the product of two diffeomorphisms in $\mathrm{Diff}_{0}^{r}(\mathbb{R} \times (-1,1))$ $\epsilon'$-close to the identity is $\epsilon$-close to the identity. We will first find a diffeomorphism $g_{1}$ $\epsilon'$-close to the identity which preserves the foliation with leaves $\mathbb{S}^{1} \times \left\{y \right\}$ and for which, for $\epsilon_{2}>0$ small enough, $g_{1}^{-1}\circ f$ preserves the germ at $\mathbb{S}^{1} \times \left\{0\right\}$ of the foliation with leaves $\left\{x \right\} \times (-1, 1) $. Then, we will find a diffeomorphism $g_{2}$ $\epsilon'$-close to the identity which equals to $g_{1}^{-1}\circ f$ on a neighbourhood of $\mathbb{S}^{1} \times \left\{0\right\}$ and take $g=g_{1} \circ g_{2}$.

Take $\delta_{1}<\epsilon_{1}$ such that $f^{-1}$ is well defined on $\mathbb{S}^{1}\times (-\delta_{1},\delta_{1})$ and, if the expression of $f^{-1}$ in coordinates is given by:
$$\begin{array}{rcl}
f^{-1}:\mathbb{S}^{1} \times (-\delta_{1},\delta_{1}) & \rightarrow & \mathbb{S}^{1}\times(-1,1) \\
(x,y) & \mapsto & (u(x,y),v(x,y))
\end{array}
,$$
then:
$$\forall (x,y) \in \mathbb{S}^{1}\times (-\delta_{1},\delta_{1}),
\left\{
\begin{array}{l}
|u(x,y)-x|<\epsilon' \\
\frac{\partial u}{\partial x}(x,y)>0
\end{array}
\right.
.$$
Denote by $u_{y}^{-1}$ the inverse of $u(.,y)$.
Denote by $\lambda:(-1,1)\rightarrow[0,1]$ a $C^{\infty}$ map which is supported in $(-\delta_{1},\delta_{1})$ and is equal to one on a neighbourhood of $0$. Now, define:
$$\begin{array}{rcl}
g_{1}:\mathbb{S}^{1}\times (-1,1) \ \rightarrow \ \mathbb{S}^{1}\times (-1,1) \\
(x,y) \ \mapsto \ (\lambda(y)u_{y}^{-1}(x) + (1-\lambda(y))x,y)
\end{array}
.$$
In other words, when $y$ is close to $0$, $g_{1}(x,y)$ is the point with ordinate $y$ which meets $f(\left\{x\right\}\times (-\delta_{1},\delta_{1}))$.
For $\epsilon_{2}>0$ small enough, the map $g_{1}$ satisfies then the following properties:
\begin{enumerate}
\item $\forall x \in \mathbb{S}^{1}, g_{1}^{-1} \circ f (\left\{x\right\} \times (-\epsilon_{2}, \epsilon_{2})) \subset \left\{x\right\} \times (-1,1)$.
\item $d(g_{1},Id)<\epsilon$.
\item $g_{1 \ |I\times(-1,1)}=Id$.
\end{enumerate}

Now the expression of the embedding $f_{1}=g_{1}^{-1} \circ f$ in coordinates is of the form:
$$\begin{array}{rcl}
f_{1}:\mathbb{S}^{1} \times (-\epsilon_{2},\epsilon_{2}) & \rightarrow & \mathbb{S}^{1}\times(-1,1) \\
(x,y) & \mapsto & (x,w(x,y))
\end{array}
.$$
Consider the isotopy, defined for $t$ in $[0,1]$:
$$\begin{array}{rcl}
f_{t}:\mathbb{S}^{1} \times (-\epsilon_{1},\epsilon_{1}) & \rightarrow & \mathbb{S}^{1}\times(-1,1) \\
(x,y) & \mapsto & (x,(1-t)y+t w(x,y))
\end{array}
.$$
The vector field associated to this isotopy is:
$$X(t,x,y)=\frac{d}{dt}(f_{t})(f_{t}^{-1}(x,y)).$$
It is defined for all $t$ in a time-independant neighbourhood of $\mathbb{S}^{1} \times \left\{0\right\}$ which we denote by $\mathbb{S}^{1} \times (-\delta_{2},\delta_{2})$. Observe that this vector fields vanishes on $\mathbb{S}^{1} \times \left\{0\right\}$ and on $I\times (-\delta_{2},\delta_{2})$. Take $0<\delta_{3}<\delta_{2}$ such that:
$$\sup_{(t,x,y)\in[0,1] \times \mathbb{S}^{1} \times (-\delta_{3},\delta_{3})} \left\| X(t,x,y) \right\| < \epsilon'.$$
Now, let $\lambda:\mathbb{R}\rightarrow [0,1]$ be a $C^{\infty}$ function supported in $(-\delta_{3},\delta_{3})$ which is equal to $1$ on a neighbourhood of $0$. Define a vector field $Y$ on $\mathbb{S}^{1}\times(-1,1)$ for $t$ in $[0,1]$ by:
\begin{itemize}
\item $Y(t,x,y)=\lambda(y)X(t,x,y)$ if $(x,y)\in \mathbb{S}^{1}\times (-\delta_{3},\delta_{3})$.
\item $Y(t,x,y)=0$ otherwise.
\end{itemize}
Then, if we denote by $g_{2}$ the time $1$ of the flow of $Y$, the diffeomorphism $g_{2}$ satisfies the following properties:
\begin{enumerate}
\item $g_{2}=f_{1}$ in a neighbourhood of $\mathbb{S}^{1}\times \left\{0\right\}$.
\item $d(g_{2},Id)<\epsilon'$.
\item $g_{2 \ |I\times(-1,1)}=Id$.
\end{enumerate}
Now, the diffeomorphism $g=g_{1} \circ g_{2}$ satisfies the required properties.
\end{proof}

Now, we are ready to prove Lemma \ref{extensionlemma} in the case where $K \cap \partial M=\emptyset$.

\begin{proof}[Proof of Lemma \ref{extensionlemma} in the case where $K \cap \partial M=\emptyset$]
Consider a tubular neighbouhood of $\partial K$ sufficiently small so that, after identification of this neighbourhood with $\partial K \times(-1,1) \subset U$, there exists $I$, a finite union of intervals of $\partial K$, such that $F \subset I \times (-1,1) \subset N$. Choose $\epsilon'>0$ so that the product of two diffeomorphisms in $\mathrm{Diff}^{r}_{0}(M)$ $\epsilon'$-close to the identity for the $C^{0}$ distance $d$ is $\epsilon$-close to the identity. By applying successively Lemmas \ref{extembcirc} and \ref{extembann} to a diffeomorphism $f$ sufficiently close to the identity, we find a diffeomorphism $g_{1}$ with support included in $U$ which satisfies:
$$\left\{ \begin{array}{l}
d(g_{1},Id)<\epsilon' \\
d(g_{1}^{-1} \circ f, Id)< \epsilon'
\end{array}
\right.
,$$
and which coincides with $f$ in a neighbourhood of $\partial K$. Denote by $g_{2}$ the diffeomorphism which is equal to $g_{1}^{-1}\circ f$ on $K$ and to the identity elsewhere. The diffeomorphism $g=g_{1}\circ g_{2}$ satisfies then the required properties.
\end{proof}

To prove the case $K\cap \partial M \neq \emptyset$ in Lemma \ref{extensionlemma}, we need a few extra lemmas.

\begin{lemma} \label{extembint}
Let $J$ be an open interval in $\mathbb{R}$, $J'$ be a closed subinterval of $J$ and $I$ be a finite union of closed intervals in $\mathbb{R}$. For any $\epsilon>0$, there exists $\alpha>0$ such that, given a $C^{r}$ increasing embedding $f:J \rightarrow \mathbb{R} \times \left\{0\right\} \subset \mathbb{R}\times [0,1)$ with:
\begin{itemize}
\item $d(f,Id)<\alpha$
\item $f_{|I \cap J}=Id$,
\end{itemize}
there exists a diffeomorphism $g$ in $\mathrm{Diff}_{0}^{r}(\mathbb{R}\times [0,1))$ which satisfies:
\begin{enumerate}
\item $g_{|J'\times\left\{0\right\}}=f_{|J'\times \left\{0\right\}}$.
\item $d(g,Id)<\epsilon$.
\item $g_{|I\times [0,1)}=Id$.
\end{enumerate}
\end{lemma}

\begin{proof}
Let us take $\alpha<min(\epsilon,d(J',J^{c}))$, where $J^{c}$ is the complementary of $J$ in $\mathbb{R}$. Take an embedding $f$ as in the hypothesis of the lemma. Consider the isotopy:
$$\begin{array}{rcl}
h:[0,1] \times J \times [0,1) & \rightarrow & \mathbb{R} \times [0,1) \\
(x,y) & \mapsto & ((1-t)x+tf(x),y)
\end{array}
$$
and denote by $h_{t}$ the embedding $h(t,.)$. The vector field associated with this isotopy is:
$$X(t,x,y)=\frac{d(h_{t})}{dt}(h_{t}^{-1}(x,y)).$$
Notice that $X(t,.)$ is defined on $h_{t}(J \times[0,1))$ which contains a neighbourhood of $J'\times [0,1)$ by the choice of $\alpha$. Then, by multiplying $X$ by a $C^{\infty}$ function supported in a neighbourhood of $J'\times [0,\frac{1}{2}]$ and which is equal to one on a smaller neighbourhood of $J'\times [0,\frac{1}{2}]$, we obtain a compactly supported vector field on $\mathbb{R} \times[0,1)$ whose time one flow satisfies the required properties.
\end{proof}

The following lemma is analogous to Lemma \ref{extembann}.

\begin{lemma} \label{extembhalfplane}
Let $\epsilon_{1}\in [0,1)$, $r\in \mathbb{N} \cup \left\{\infty\right\}$, $J$ be an open interval in $\mathbb{R}$, $J'$ be a closed interval included in $J$ and $I$ denote a finite union of closed intervals in $\mathbb{R}$. Denote by $f$ a $C^{r}$-embedding of $J\times [0,\epsilon_{1})$ in $\mathbb{R}\times [0,1)$ which fixes pointwise $J \times \left\{0\right\}$ and $I\cap J\times [0,\epsilon_{1})$. Then, for any $\epsilon>0$, there exists a diffeomorphism $g$ in $\mathrm{Diff}_{0}^{r}(\mathbb{R} \times [0,1))$ fixing pointwise $\mathbb{R} \times \left\{0\right\}$ which satisfies:
\begin{enumerate}
\item $g=f$ in a neighbourhood of $J'\times\left\{0\right\}$.
\item $d(g,Id)<\epsilon$.
\item $g_{|I\times [0,1)}=Id$.
\end{enumerate}
\end{lemma}

\begin{proof}
This proof is almost the same as the proof of Lemma \ref{extembann}.
Consider $\epsilon'>0$ such that the product of two diffeomorphisms in $\mathrm{Diff}_{0}^{r}(\mathbb{R} \times |0,1))$ $\epsilon'$-close to the identity is $\epsilon$-close to the identity. We will first find a diffeomorphism $g_{1}$ $\epsilon'$-close to the identity which preserves the foliation with leaves $\mathbb{R} \times \left\{y \right\}$ and for which, for $\epsilon_{2}>0$ small enough and $x$ in a neighbourhood of $J'$, $g_{1}^{-1}\circ f(\left\{x \right\} \times [0, \epsilon_{2})) \subset \left\{x \right\} \times (-1,1)$. Then we will find a diffeomorphism $g_{2}$ $\epsilon'$-close to the identity which equals $g_{1}^{-1}\circ f$ on a neighbourhood of $J' \times \left\{0\right\}$ and take $g=g_{1} \circ g_{2}$. 
Suppose the expression of $f^{-1}$ in coordinate is given by:
$$\begin{array}{rcl}
f^{-1}:J \times [0,\epsilon_{1}) & \rightarrow & \mathbb{R}\times(-1,1) \\
(x,y) & \mapsto & (u(x,y),v(x,y))
\end{array}
.$$
Take $\delta<\epsilon_{1}$ and an open interval $J''$ with $J' \subset J'' \subset \overline{J''} \subset J$ such that:
$$\forall (x,y) \in J'' \times [0,\delta),
\left\{
\begin{array}{l}
|u(x,y)-x|<\epsilon' \\
\frac{\partial u}{\partial x}(x,y)>0
\end{array}
\right.
.$$
Denote by $u_{y}$ the function $u(.,y)$.
Let $\lambda: \mathbb{R} \rightarrow [0,1]$ be a $C^{\infty}$ function supported in $(-\delta,\delta)$ and which is equal to one in a neighbourhood of $0$. Consider the embedding:
$$\begin{array}{rcl}
h:J'' \times [0,1) & \rightarrow & \mathbb{R}\times[0,1) \\
(x,y) & \mapsto & \left\{ \begin{array}{l}((1-\lambda(y))y+\lambda(y)u_{y}^{-1}(x),y) \ \mathrm{if} \ y\leq \delta \\ (x,y) \ \mathrm{otherwise} \end{array} \right.
\end{array}
.$$
It coincides with $f$ on a neighbourhood of $J''\times \left\{0\right\}$. We now have to extend it in the horizontal direction in order to have a diffeomorphism. To do this, we consider the isotopy (for $t$ in [0,1]):
$$\begin{array}{rcl}
h_{t}: J'' \times [0,1) & \rightarrow & J''\times[0,1) \\
(x,y) & \mapsto & ((1-t)x+tu'(x,y),y)
\end{array}
,$$
where $u'$ is the first coordinate of $h$. It suffices then to cut the vector field generating this isotopy by a $C^{\infty}$ function to obtain a compactly supported vector field defined on $\mathbb{R}\times [0,1)$. Then the time one of the flow of this vector field gives a diffeomorphism $g_{1}$ of $\mathbb{R} \times[0,1)$ which satisfies the following properties:
\begin{enumerate}
\item for $x$ in a neighbourhood of $J'$, $g_{1}^{-1} \circ f (\left\{x\right\} \times [0, \epsilon_{2})) \subset \left\{x\right\} \times [0,1)$.
\item $d(g_{1},Id)<\epsilon'$.
\item $g_{1 \ |I\times[0,1)}=Id$.
\end{enumerate}
The construction of $g_{2}$ is done exactly the same way as in Lemma \ref{extembann} : consider a well-chosen isotopy which joins the identity to $f_{1}=g_{1}^{-1} \circ f$, the vector field generating it and cut this vector field outside a neighbourhood of $J'\times \left\{0\right\}$ to get a global compactly supported vector field. Take the time one of the flow of this vector field.
Then, it suffices to take $g=g_{1} \circ g_{2}$.
\end{proof}

Now, we can finish the proof of Lemma \ref{extensionlemma} by treating the case where $K \cap \partial M \neq \emptyset$.

\begin{proof}[Proof of Lemma \ref{extensionlemma}: case where $K \cap \partial M \neq \emptyset$]
We will just explain what has to be taken care of in this proof which is analogous to the proof of the case where $\partial K \cap M= \emptyset$.
Take the diffeomorphism $f$ sufficiently close to the identity.
By using Lemma \ref{extembint}, we may find a diffeomorphism $g_{1}$ $C^{0}$-close to the identity which coincides with $f$ on a neighbourhood in $\partial M$ of $\partial K \cap \partial M$. Then, we use Lemma \ref{extembhalfplane} to get a diffeomorphism $g_{2}$ $C^{0}$-close to the identity which coincides with $g_{1}^{-1} \circ f$ on a neighbourhood $V$ in $M$ of $\partial K \cap \partial M$. Let $W$ be a neighbourhood of $\partial K \cap \partial M$ whose closure is included in $V$. Consider then a simple closed path $\gamma: \mathbb{S}^{1}=[0,1]/\mathbb{Z} \rightarrow M$ which satisfies:
$$\left\{ 
\begin{array}{l}
\gamma([-\frac{1}{4},\frac{1}{4}]) \supset \partial K - W \\
\gamma(\mathbb{S}^{1})\subset (\partial K \cup W)-\partial M
\end{array}
\right.
.
$$
By applying the first case of Lemma \ref{extensionlemma} with $K'$ the interior of the Jordan curve $\gamma$ and $F'$ the union of $F$ with the closure of $W$, we may find a $C^{r}$-diffeomorphism $g_{3}$ $C^{0}$-close to the identity which is equal to $g_{1}^{-1}\circ f$ in a neighbourhood of $K$. Then we take $g=g_{1} \circ g_{2} \circ g_{3}$ to end the proof.
\end{proof}

\end{document}